\documentclass[10pt]{article}

\makeindex

\input xy
\xyoption{all}

\usepackage{mathtools}
\usepackage{amsmath,amsfonts,amssymb,amsthm}
\usepackage{latexsym, xspace, enumerate}
\usepackage[mathscr]{eucal}
\usepackage{hyperref}
\usepackage{amscd}
\usepackage{xcolor}

\newtheorem{theorem}{Theorem}[section]

\newtheorem{claim}[theorem]{Claim}

\newtheorem{proposition}[theorem]{Proposition}
\newtheorem{lemma}[theorem]{Lemma}

\newtheorem{corollary}[theorem]{Corollary}

\theoremstyle{definition}

\newtheorem{definition}[theorem]{Definition}

\newtheorem{example}[theorem]{Example}
\newtheorem{remark}[theorem]{Remark}%[section]
\DeclareMathOperator{\supp}{\mathrm{supp}}
\newcommand{\N}{\mathbb{N}}
\newcommand{\T}{\mathbb{T}}

\newcommand{\Z}{\mathbb{Z}}
\newcommand{\Q}{\mathbb{Q}}
\newcommand{\R}{\mathbb{R}}
\newcommand{\I}{\mathbb{I}}

\newcommand{\asub}{\subseteq^{\mathcal I}}
\newcommand{\lsub}{\subseteq_{\mathcal I}}

\def\uu{\mathbf{u}}

\renewcommand{\I}{\mathcal I}
\newcommand{\F}{\mathcal F}

\newcommand{\A}{\mathcal A}
\newcommand{\B}{\mathcal B}
\newcommand{\D}{\mathcal D}
\renewcommand{\P}{\mathcal P}
\renewcommand{\SS}{\mathcal S}

\newcommand{\bx}{\mathrm{(b}_x\mathrm{)}}

\newcommand{\ax}{\mathrm{(a}_x\mathrm{)}}
\newcommand{\axuno}{\mathrm{(a1}_x\mathrm{)}}
\newcommand{\axdue}{\mathrm{(a2}_x\mathrm{)}}

\newcommand{\ix}{\mathrm{(i}_x\mathrm{)}}
\newcommand{\iix}{\mathrm{(ii}_x\mathrm{)}}

\newcommand{\unox}{\mathrm{(1}_x\mathrm{)}}

\newcommand{\duex}{\mathrm{(2}_x\mathrm{)}}

\newcommand{\trex}{\mathrm{(3}_x\mathrm{)}}

\newcommand{\Ix}{\mathrm{(I}_x\mathrm{)}}
\newcommand{\IIx}{\mathrm{(II}_x\mathrm{)}}

\newcommand{\DL}{\mathrm{(DL)}}

\newcommand{\DLu}{\mathrm{(DL_\uu)}}

\numberwithin{equation}{section}

\newcommand{\NB}%{}
{${\clubsuit}\;$}
\newcommand{\NBA}%{}
{${\spadesuit}$}

\usepackage{vmargin}
\setmarginsrb{18mm}{10mm}{18mm}{15mm}%
        {5mm}{6mm}{5mm}{9mm}
\newlength{\bibitemsep}
\newlength{\bibparskip}
\let\oldthebibliography\thebibliography
\renewcommand\thebibliography[1]{%
  \oldthebibliography{#1}%
  \setlength{\parskip}{\bibitemsep}%
  \setlength{\itemsep}{\bibparskip}%
}

\usepackage{color}

\title{Element-wise description of the $\I$-characterized subgroups of the circle}

\author{R. Di Santo\\ \href{mailto:raffaele.disanto@uniud.it}{\scriptsize raffaele.disanto@uniud.it}  
\and D. Dikranjan \\ \href{mailto:dikran.dikranjan@uniud.it}{\scriptsize dikran.dikranjan@uniud.it}
\and A. Giordano Bruno \footnote{The third named author is member of and thanks the National Group for Algebraic and Geometric Structures, and Their Applications (GNSAGA - INdAM)} \\ \href{mailto:anna.giordanobruno@uniud.it}{\scriptsize anna.giordanobruno@uniud.it}
\and H. Weber \\ \href{mailto:hans.weber@uniud.it}{\scriptsize hans.weber@uniud.it}}

\date{}

\begin{document}
\maketitle

\begin{abstract}
%Recently, generalizations of the notion of characterized subgroup of $\T$ were introduced, based on types of convergence weaker than the usual one, starting from statistical convergence and ending with the notion (due to Cartan) of $\I$-convergence for an ideal $\I$ of $\N$: 
According to Cartan, given an ideal $\I$ of $\N$, a sequence $(x_n)_{n\in\N}$ in the circle group $\T$ is said to {\em $\I$-converge} to a point $x\in \T$ if $\{n\in \N: x_n \not \in U\}\in \I$ for every neighborhood $U$ of $x$ in $\T$. For a sequence $\uu=(u_n)_{n\in\N}$ in $\Z$, let %a subgroup $H$ of $\T$ is {\em $\I$-characterized} if $$H= 
$$t_\uu^\I(\T) :=\{x\in \T: u_nx \ \text{$\I$-converges to}\ 0 \}.$$ 
For an analytic free $P$-ideal $\I$, this set is a Borel (hence, Polishable) subgroup of $\T$ with many nice properties, largely studied in the case when $\I = \F in$ is the ideal of all finite subsets of $\N$ (so $ \F in$-convergence coincides with the usual one) for its remarkable connection to topological algebra, descriptive set theory and harmonic analysis. We give a complete element-wise description of $t_\uu^\I(\T)$ when $\uu$ is a strictly increasing sequence of positive integers with $u_0=1$ and $u_n\mid u_{n+1}$ for every $n\in\N$ and under suitable hypotheses on $\I$. 
In the special case when $\I =\F in$, we obtain an alternative proof of a simplified version of a known
result from \cite{DI}. %describing the elements of a given characterized subgroup of $\T$, with an alternative proof.
\end{abstract}

\medskip
\hrule

\smallskip
\noindent {\scriptsize Keywords: topologically torsion element, characterized subgroup, $\I$-convergence, statistical convergence, arithmetic sequence.\\
MSC2020: 54H11, 40A35, 20K45, 22A10, 11J71.}

%\section*{Copia destinata a sostituire la versione sbagliata su {\tt arXiv} 
%Competing Interests - Disclosure Statement 
%}

%There are no competing interests to declare.There are no financial or non-financial competing interests to report.

\footnotesize
\tableofcontents

\normalsize
\section{Introduction}

Let $\T=\R/\Z$ be the circle group denoted additively, and let $\varphi\colon \R\to\T$, $x\mapsto \bar x:=\varphi(x)$, be the canonical projection. 
%Then the restriction $\varphi_{\restriction [0,1)}\colon[0,1)\to \T$ is a bijection. 
Following \cite{BDS}, a subgroup $H$ of $\T$ is said to be \emph{characterized} by a sequence $\uu = (u_n)$ of integers if $H=\{x\in\T: u_nx\to 0\}$. The introduction of these subgroups was motivated by problems arising from various areas of Mathematics. %so they were thoroughly investigated (e.g., see the survey \cite{DDG}).
We give below a brief overview of the history of these subgroups related to the topic of the present paper; for more detail, see the surveys \cite{DDG,D1}.

\subsection{Historical background}

The roots of this topic are in the following notions appearing in Braconnier~\cite{Bra}, Vilenkin~\cite{Vi}, Robertson~\cite{Ro}, Armacost~\cite{Arm}:
an element $x$ of a topological abelian group $G$ is:
\begin{itemize}
  \item[(a)] {\em topologically $p$-torsion}, for a prime $p$, if $p^nx\to 0$;
  \item[(b)] {\em topologically torsion} if $n!x\to 0$.
 \end{itemize}
For a prime $p$, the \emph{topologically $p$-torsion subgroup} of $G$ is ${t_{\underline p}(G):=\{x\in G: p^n x\to 0\}.}$ Similarly, the \emph{{topologically torsion subgroup}} of $G$ is ${G!:=\{x\in G: n! x\to 0\}.}$ These notions generalize the classical ones of $p$-torsion subgroup $t_p(G)$ for a prime $p$, and torsion subgroup $t(G)$ of an abelian group $G$. Clearly,  $t_p(G)\subseteq t_{\underline p}(G)$ for every prime $p$, and $t(G)\subseteq G!$.

\begin{example}\label{Exa:arm}
Armacost \cite{Arm} observed the following non-trivial facts: 
\begin{enumerate}[(a)]
  \item $t_{\underline p}(\T)=t_p(\T)=\Z(p^\infty)$ for every prime $p$;
  \item $\bar e\in \T!$, but $\bar e\not\in t(\T)=\Q/\Z$, where $e=\lim_{n\to +\infty} \left(1 + \frac{1}{n}\right)^n\in \R$ is the Euler number.
\end{enumerate}
\end{example}

Example~\ref{Exa:arm}(b) shows that the topologically torsion subgroup $\T!$ may be much more complicated compared to the torsion subgroup $\Q/\Z$ of $\T$. This is why Armacost \cite{Arm} posed the problem to describe the subgroup $\T!$ of $\T$.

\smallskip
The following definition generalizes the above ones of topological $p$-torsion and topological torsion.

\begin{definition}[{\cite{DPS,D1}}]\label{def:top:tor}
For a sequence $\uu=(u_n)$ in $\Z$ and a topological abelian group $G$, an element $x\in G$ is
\emph{{topologically $\uu$-torsion}} if $u_n x\to 0$ in $G$. The \emph{{topologically $\uu$-torsion subgroup}} of $G$ is
${t_\uu(G):=\{x\in G: u_n x\to 0\}}$.
\end{definition}

\begin{enumerate}[(a)]
  \item For a prime $p$, $t_{\underline p}(G)=t_\mathbf p(G)$, where $\mathbf p=(p^n)$;
  \item $G!=t_\mathbf n(G)$, where $\mathbf n=((n+1)!)$.
\end{enumerate}

Clearly, for $G=\T$, the topologically $\uu$-torsion subgroup $t_\uu(\T)$ is the subgroup of $\T$ characterized by $\uu$.

\smallskip
Definition \ref{def:top:tor} was given in \cite[\S 4.4.2]{DPS} only for arithmetic sequences, as $\mathbf p$ and $\mathbf n$ share this property: recall that a sequence $\uu = (u_n)$ of integers is \emph{arithmetic} if it is strictly increasing, $u_0=1$ and $u_n\mid u_{n+1}$ for every $n\in\N$. Let $\A$ denote the family of all arithmetic sequences. 
For $\uu\in\A$, let $b^\uu_n:={u_{n}}/{u_{n-1}}$ for every $n\in\N_+$ and $\mathbf b^\uu:=(b_n^\uu)$; when it is clear in the context we simply write $b_n=b^\uu_n$ and $\mathbf b=\mathbf b^\uu$. In these terms $$u_1 = b_1,\ u_2= b_1b_2,\ \ldots,\ u_{n+1}= b_{n+1}b_n \ldots b_1 = b_{n+1}u_n,\ \ldots.$$ 

Each arithmetic sequence $\uu$ gives rise to a nice representation: for every $x\in [0,1)$, there exists a unique sequence $(c_n^\uu(x))_{n\in\N_+}$ in $\N$ such that
\begin{equation}\label{ex-4}
{x= \sum_{n=1}^\infty\frac{c_n^\uu(x)}{u_n},}
\end{equation}
with $c_n^\uu(x)<b^\uu_n$ for every $n\in\N_+$, and $c_n^\uu(x)<b^\uu_n-1$ for infinitely many $n\in\N_+$.
When no confusion is possible, we shall write $c_n^\uu$, $c_n(x)$ or simply  $c_n$ in place of $c_n^\uu(x)$. For $x\in[0,1)$, with canonical representation~\eqref{ex-4}, let 
$$\supp(x):=\{n\in\N_+: c_n\neq0\}\quad\text{and}\quad  \supp_b(x):=\{n\in\N_+: c_n=b_n-1\}.$$ 
Clearly, $\supp_b(x)\subseteq\supp(x)$ and $\supp_b(x)$ cannot be cofinite by definition. 

\smallskip
The following solution to Armacost problem was obtained independently by Borel \cite{Bo2}, and somewhat earlier by  Dikranjan, Prodanov and Stoyanov \cite{DPS}.

\begin{theorem}[{\cite{Bo2,DiD}}]\label{Prob:A} Let $x\in[0,1)$. Then $\bar x\in\T!=t_\mathbf n(\T)$ if and only if $\varphi\left(\frac{c_n}{n+1}\right)\to 0$.
\end{theorem}

On the other hand, the torsion subgroup $\Q/\Z$ of $\T$ can be obtained as a topologically $\uu$-torsion subgroup as follows:

\begin{example}[\cite{DK,Bogdanovic}]\label{Exa:DK} 
With $\uu$ the sequence $({1!},{2!}, 2\cdot 2!, {3!}, 2\cdot 3!, 3\cdot 3!, {4!}, \ldots, {n!}, 2\cdot n!, 3\cdot n!, \ldots, n\cdot n!, {(n+1)!}, \ldots)$, we have $\Q/\Z=t_\uu(\T)$. Similarly, one can characterize arbitrary subgroups of $\Q/\Z$. 
\end{example}

Following~\cite{DI}, fixed $\uu\in\A$, call an infinite  subset $A$ of $\N$ \emph{$b$-bounded} (resp., \emph{$b$-divergent}) if the sequence $(b_n)_{n\in A}$ is bounded (resp., diverges to infinity). In particular, call $\uu$ \emph{$b$-bounded} (resp.,  \emph{$b$-divergent}) if $\N_+$ is $b$-bounded (resp., $b$-divergent).

\smallskip
In the general setting, we start recalling the following clear description of the topologically $\uu$-torsion elements when the arithmetic sequence $\uu$ is either $b$-bounded or $b$-divergent. Item (a) follows from~\cite[Corollary~2.4]{DiD} and~\cite[\S 4.4.2, Theorem~(b)]{DPS}, while item (b) is \cite[Corollary~3.7]{DI} and generalizes Theorem~\ref{Prob:A}.

\begin{theorem}\label{Thm:DPS} 
%\NB\footnote{Ho tolto qui \cite{DPS}, che non ha un risultato esplicito da presentare come
%(a) \& (b), ma contiene esplictamente due punti (b) \& (c) che sono simili al punto (a) sopra, ma non proprio uguali. Perciò, per massima sicurezza/correttezza 
%non metterei un riferiemto a \cite{DPS} per quanto riguarda QUESTO teorema e l'articolo attuale che stiamo scrivendo.
%Ma nella futura tesi di Raffaele queste cose vanno spiegate bene, perché \cite{DPS} è attualmente introvabile e un lettore estraneo non può andare a verificare.
%Dare informazioni fuorvianti no va bene. Spero che questa versione si avvicini ad una descrizione più fedele della situazione, ma probabilmente si possono apportare miglioramenti vari che sollecito caldamente (come anche leggere le fonti originali ("classici" :-).}
Let $\uu\in\A$ and $x\in[0,1)$.
\begin{enumerate}[(a)]
\item  If $\uu$ is $b$-bounded, then $\bar x\in t_\uu(\T)$ if and only if $(c_n)\ \text{is eventually}\ 0.$
\item If $\uu$ is $b$-divergent, then $\bar x\in t_\uu(\T)$ if and only if $\varphi\left(\frac{c_n}{b_n}\right)\to 0$.
\end{enumerate}
\end{theorem}

The general result \cite[\S 4.4.2, Theorem]{DPS} proposing a description of $t_\uu(\T)$ for arbitrary $\uu\in\A$ 
in terms of the limit $\frac{c_n}{b_n}\to 0$ has a gap, which led to an incomplete solution of Armacost's problem (due to the fact that for $\uu=\mathbf n$ the factorial sequence $((n+1)!) $ this gives only the stronger condition $\frac{c_n}{n+1}\to 0$, missing in this way the elements $\bar x\in t_\mathbf n(\T)$ with $\varphi\left(\frac{c_n}{n+1}\right)\to 0$, but $\frac{c_n}{n+1}\not \to 0$; see item (b) in Theorem~\ref{Thm:DPS} above). This gap was filled in~\cite[Theorem 3.1]{D1}, whose proof appeared only later in~\cite[Theorem~2.2]{DiD}. Unfortunately, in~\cite{DI} a gap was found also in~\cite[Theorem~2.2]{DiD}, as well as in some of its corollaries.
Finally, in \cite[Theorem~2.3]{DI} the following complete description of the topologically $\uu$-torsion elements of $\T$ for an arbitrary arithmetic sequence $\uu$ was obtained. %(correcting a previous incomplete description from~\cite{DiD}, anticipated also in \cite[Theorem 3.1]{D1}).

For subsets $A,B$ of $\N$, we say that $A$ is \emph{almost contained} in $B$
whenever $|A\setminus B|< \infty$ and we write $A\subseteq^* B$.

\begin{theorem}[{\cite{DI}}]\label{DiD} Let $\uu\in\mathcal A$ and $x\in[0,1)$. Then $\bar x \in t_{\uu}(\T)$ if and only if either $\supp(x)$ is finite or, if $\supp(x)$ is infinite, then for all infinite $A\subseteq\N$ the following holds.
\begin{itemize}
   	\item[$\mathrm{(a)}$] If $A$ is $b$-bounded, then:
\begin{itemize}		
	\item[$\mathrm{(1)}$] if $A\subseteq^*\supp(x)$,  then $A+1 \subseteq^*\supp(x)$, $A\subseteq^*\supp_b(x)$ and $\lim_{n\in A}{\frac{c_{n+1}+1}{b_{n+1}}}=1$.
\\ Moreover, if $A+1$ is $b$-bounded, then $A+1\subseteq^*\supp_b(x)$ as well;
	\item[$\mathrm{(2)}$] if $A\cap\supp(x)$ is finite, then $\lim_{n\in A}{\frac{c_{n+1}}{b_{n+1}}}=0$.
	\\ Moreover, if $A+1$ is $b$-bounded, then $(A+1)\cap\supp(x)$ is finite as well.
\end{itemize}
	\item[$\mathrm{(b)}$] If $A$ is $b$-divergent, then $\lim_{n\in A}{\varphi\left(\frac{c_n}{b_n}\right)}=\lim_{n\in A}{\varphi\left(\frac{c_n+1}{b_n}\right)}=0$.
\end{itemize}
\end{theorem}

The characterization provided  in Theorem~\ref{DiD} simplifies essentially imposing on $\uu$ the following additional property  (see Corollary \ref{splitting}).

\begin{definition}[{\cite[Definition 3.10]{DI}}]\label{splittingdef}
Let $\uu\in\A$. An infinite set $A\subseteq \N_+$ is said to have the {\em $\uu$-splitting property} if there exists a partition $A=B \sqcup D$ such that: 
\begin{itemize}
    \item[(a)] $B$ and $D$ are either empty or infinite;
    \item[(b)] if $B$ is infinite, then $B$ is $b$-bounded;
    \item[(c)] if $D$ is infinite, then $D$ is $b$-divergent. %there exists $D' \subseteq \N$ with $D =^{\I} D'$ such that $D'$ is $b$-divergent. 
\end{itemize} 
We say that  $\uu\in\A$ has the {\em splitting property} if $\N_+$ has the $\uu$-splitting property.
\end{definition}

\begin{corollary}[{\cite[Corollary 3.13]{DI}}]\label{splitting} Suppose that $\uu\in\A$ has the splitting property witnessed by the partition $\N_+=B\sqcup D$, and let $x\in[0,1)$ with $S=\supp(x)$ and $S_b=\supp_b(x)$. Then $\bar x\in t_\uu(\T)$ if and only if:
\begin{itemize}
\item[(1)] $(B\cap S)+1 \subseteq^* S$, $B\cap S \subseteq^* S_b$ and if $B\cap S$ is infinite then $\lim_{n \in B\cap S}\frac{c_{n+1}+1}{b_{n+1}}=1$; 
   \item[(2)] if $B\setminus S$ is infinite, then $\lim_{n \in B\setminus S}\frac{c_{n+1}}{b_{n+1}}=0$;
   \item[(3)] if $D\cap S$ is infinite, then $\lim_{n \in D\cap S} \varphi\left(\frac{c_n}{b_n}\right)=0$.
\end{itemize} 
\end{corollary}

Clearly, one can deduce Theorem~\ref{Thm:DPS} from Corollary~\ref{splitting}.

\subsection{$\I$-characterized subgroups of $\T$}\label{I-conv}

In the last few years some new trends appeared in the field of characterized subgroups, based on weaker notions of convergence. See~\cite{BCsurvey} for a recent survey on the topic.
%See \cite{DG1,DasGhosh,Medit2024,BLMS2025,Expo2025,BSM2025,DDB} and see also~\cite{BCsurvey} for a recent survey on the topic. It is worth mentioning here also the very recent~\cite{Leo}.

\smallskip
For $A\subseteq\N$, let $A(n):=\{i\in A: i\leq n\}=A\cap[0,n].$ Following \cite{BDP} (see~\cite{BDH}), fixed $\alpha\in(0,1]$, the \emph{upper natural density of order $\alpha$} is 
$$
\overline d_\alpha(A):=\limsup_{n\to\infty}\frac{|A(n)|}{n^\alpha}.
$$
For $\alpha=1$, $d:= \overline d_1$ is the classical upper natural density. As in~\cite{BDH}, a sequence $(x_n)$ in $\T$ is said to
\emph{$\alpha$-statistically converge} to $x\in \T$, denoted by $x_n\overset{s_\alpha}\longrightarrow x$, if for every neighborhood $U$ of $x$ in $\T$, $d_\alpha(\{n\in\N: x_n\not\in U\})=0$. Moreover, for a sequence of integers $\uu$ and $\alpha\in(0,1]$, let 
$$
t_\uu^\alpha(\T):=\{x\in\T: u_nx\overset{s_\alpha}\longrightarrow 0\}.
$$
A subgroup $H$ of $\T$ is \emph{$\alpha$-statistically characterized} if there exists a sequence of integers $\uu$ such that $H=t_\uu^\alpha(\T)$; the elements of $t_\uu^\alpha(\T)$ are called \emph{topologically $\uu_\alpha$-torsion elements of $\T$}. 
The $1$-statistically characterized subgroups $t_\uu^s(\T):=t_\uu^1(\T)$ of $\T$ were introduced and studied in the seminal paper~\cite{DDB} under the name statistically characterized subgroups.

%\smallskip
An element-wise description, in the spirit of Theorem~\ref{DiD}, of $t_\uu^\alpha(\T)$ for $\uu\in\mathcal A$ is given in~\cite{DG1}, but the proof of the theorem presents a gap (see Remark~\ref{implies*}).

\smallskip
The $\alpha$-statistical convergence and the usual one are particular cases of the ideal convergence, introduced by Cartan~\cite{Cartan} as follows. %(see \S\ref{idealconv-sec}). 
A non-empty subfamily $\I$ of $\P(\N)$ is an \emph{ideal of $\N$} if  $A\cup B\in \I$ for every $A,B\in\I$, and if $A \in \P(\N)$ and $B\in\I$ with $A \subseteq B$, then $A\in\I$. An easy example of ideal of $\N$ is the family $\F in$ of all finite subsets of $\N$, and we always assume that an ideal $\I$ contains $\F in$, that is, $\I$ is \emph{free}.

For a free ideal $\I$ of $\N$, a sequence $(x_n)$ in $\T$ is said to {\em $\I$-converge} to a point $x\in \T$, denoted by $x_n\overset{\I}\longrightarrow x$, if $\{n\in \N: x_n \not \in U\}\in \I$ for every neighborhood $U$ of $x$ in $\T$.  Moreover, for a sequence $\uu$ of integers, let 
$$t_\uu^\I(\T) :=\{x\in \T: u_nx\overset{\I}\longrightarrow 0 \}.$$
So, a subgroup $H$ of $\T$ is {\em $\I$-characterized} if there exists a sequence $\uu$ of integers such that $H=t_\uu^\I(\T)$; the elements of $t_\uu^\I(\T)$ are called \emph{topologically $\uu_\I$-torsion elements} of $\T$.

The usual convergence is the $\F in$-convergence, so $t_\uu^{\F in}(\T)=t_\uu(\T)$. Moreover, always $t_\uu(\T)\subseteq t_\uu^\I(\T)$ as $\I$ is free.
For $\alpha\in(0,1]$, the family $\I_\alpha:=\{A\subseteq\N: d_\alpha(A)=0\}$ is an ideal of $\N$; for $\alpha=1$, $\I_d:=\I_1$.
So, the $\alpha$-statistical convergence is the $\I_\alpha$-convergence, and $t_\uu^{\I_\alpha}(\T)=t_\uu^\alpha(\T)$; in particular, $t_\uu^{\I_d}(\T)=t_\uu^s(\T)$.

\smallskip
To state our main result, recall that an ideal $\I$ of $\N$ is a \emph{$P$-ideal}  if for every countable family $\{A_n:n\in\N\}\subseteq \I$ there exists a {\em pseudounion} $U\in \I$ of $\{A_n:n\in\N\}$ (namely, a subset $U$ of $\N$ with  $A_n\subseteq^*U$ for all $n\in\N$, the pseudounion is not uniquely determined).  Moreover, following~\cite{DasGhosh}, a set $A\subseteq\N$ is \emph{$\mathcal I$-translation invariant} if $(A+n)\cap\N\in\mathcal I$ for every $n\in\Z$, while the ideal $\mathcal I$ is \emph{translation invariant} if every $A\in\mathcal \I$ is $\mathcal \I$-translation invariant. The ideals $\F in$ and $\I_\alpha$, with $\alpha\in(0,1]$, are translation invariant free $P$-ideals. For two subsets $A, B$ of $\N$ and an ideal $\I$ of $\N$, denote:
\begin{itemize}
   \item[-] $A \subseteq^{\I} B$ if $A \setminus B \in \I$;
   \item[-] $A \subseteq_{\I} B$ if $A \subseteq B$ and $B \setminus A \in \I$, i.e., $A \subseteq B$ and $B\subseteq^\I A$.
\end{itemize}

\begin{definition}\label{IdealiDuBu} For $\uu \! \in \! \A$, denote by  $\B_\uu^*$ (resp., $\D_\uu^*$) the family of all $b$-bounded (resp., $b$-divergent) subsets of $\N$.
\end{definition}

Obviously, both $\B_\uu^*$ and $\D_\uu^*$ are ideal bases, while  $\B_\uu:= \B_\uu^* \cup \F in$ and $\D_\uu:= \D_\uu^* \cup \F in$ are free ideals of $\N$. 

\smallskip
The following is the main result of this paper, which characterizes the topologically $\uu_\I$-torsion elements of $\T$ with respect to a given $\uu\in\A$.
In the sequel with $\ax$, $\axuno$, $\axdue$ and $\bx$, we always refer to the properties in the following statement.

\begin{theorem}\label{conjecture} 
Let $\uu\in\mathcal A$, let $\I$ be a translation invariant free $P$-ideal of $\N$ and $x \in[0,1)$. Then $\bar x\in t^{\I}_\uu(\T)$ if and only if for all $A \in\P(\N)\setminus\I$ the following holds.
 \begin{itemize}
   \item[$\ax$] If $A\in \B_\uu$ (i.e., $A$ is $b$-bounded), then:
 \begin{itemize}
     \item[$\axuno$] if $A \subseteq^{\I} \supp(x)$, then $A \subseteq^{\I} \supp_b(x)$, $A+1\subseteq^{\I} \supp(x)$ and there exists $A' \subseteq_{\I} A$ such that $A' \subseteq \supp_b(x)$ and $\lim_{n \in A'} \frac{c_{n+1}+1}{b_{n+1}}=1$;
     \item[$\axdue$]  if $A\cap \supp(x) \in \I$, then there exists $B'\subseteq_\I A$ such that $\lim_{n \in B'} \frac{c_{n+1}}{b_{n+1}}=0$. 
  \end{itemize}
\item[$\bx$] If $A\in \D_\uu$ (i.e., $A$ is $b$-divergent), then there exists $B \subseteq_{\I} A$ such that $\lim_{n \in B} \varphi\left(\frac{c_n}{b_n}\right)=0$.
\end{itemize}
\end{theorem} 

%\footnote{grosso taglio qui, il TeoGhosh \ref{Theorem2.9} postecipato.}
Ghosh~\cite{Ghosh} stated a similar result (see Theorem \ref{Theorem2.9}) but the proof of his theorem presents a flaw (see Remark \ref{implies*} 
for a description of that flaw and Remark \ref{implies} for a precise comparison of these two results). 
Theorem~\ref{conjecture} implies Theorem \ref{Theorem2.9} (see Remark \ref{implies}), and so also Theorems~\ref{DiD} and~\ref{Th:DG}, as these two are clearly special cases of Theorem~\ref{Theorem2.9}, respectively with $\I=\F in$ and $\I=\I_\alpha$.
  
%%%%%%%%%%%%%%%%%%%%%%%%%%%%%%%%%%%%%%.  C U T. H E R E %%%%%%%%%%
\smallskip
The following almost immediate corollary of Theorem~\ref{conjecture} (see \S\ref{final stage} for a proof)  is focused on a relevant choice of the pair $\uu, \I$ 
which makes the characterization in Theorem~\ref{conjecture} particularly easy and transparent.  Indeed, when $\I \supseteq \B_\uu$,   $\ax$ in Theorem~\ref{conjecture} is vacuously satisfied (as $A \not \in \I$ yields $A \not \in  \B_\uu$, i.e., $A$ cannot $b$-bounded), so the characterization is reduced to  $\bx$. 
Moreover, Corollary~\ref{IcB} allows us to obtain Theorem \ref{Thm:DPS}(b) and Theorem \ref{Prob:A} (with $\I = \F in$, and consequently, $\N\in \D_\uu=\P(\N)$). 

\begin{corollary}\label{IcB}
Let $\uu\in\A$, let $\I$ be a translation invariant free $P$-ideal of $\N$ with $\B_\uu\subseteq\I$ and  $x\in[0,1)$. Then $\bar x\in t_\uu^{\I}(\T)$ if and only if $\bx$ holds.
\end{corollary}

\smallskip
In Definition~\ref{Isp} we recall \cite[Definition 3.1]{Ghosh}, which is inspired by the notion of splitting sets mentioned above. 
For $\uu\in\A$ and a free ideal $\I$ of $\N$, we say that a subset $X$ of $\N$ is \emph{$b$-bounded mod $\I$} (resp., \emph{$b$-divergent mod $\I$}) if there exists a $b$-bounded (resp., $b$-divergent) $A\subseteq_\I X$. For $\I = \F in$, ``$b$-bounded mod $\I$" (resp., ``$b$-divergent mod $\I$") simply coincides with ``$b$-bounded" (resp., ``$b$-divergent"). Moreover, we say that $\uu$ is \emph{$b$-bounded mod $\I$} (resp., \emph{$b$-divergent mod $\I$})  if $\N$ is {$b$-bounded mod $\I$} (resp., {$b$-divergent mod $\I$}).

Given $\uu\in\A$ and a free ideal $\I$ of $\N$, a subset $A$ of $\N$ has the \emph{$\uu_\I$-splitting property} if it has the form $A=B\sqcup D$, with $B$ $b$-bounded mod $\I$ and $D$ $b$-divergent mod $\I$. Moreover, we say that \emph{$\uu$ has the $\I$-splitting property} if $\N_+$ has the $\I$-splitting property.
Equivalently: %as in \cite[Definition 3.1]{Ghosh}, $\uu$ has the $\I$-splitting property if there exists a partition $\N_+=B \sqcup D$ such that: 
%\begin{itemize}
%    \item[(1)] $B$ and $D$ are either empty or $B, D\not\in\I$;
%    \item[(2)] if $B\not\in\I$, then $B$ is $b$-bounded mod $\I$;
%    \item[(3)] if $D\not\in\I$, then $D$ is $b$-divergent mod $\I$. 
%\end{itemize}
%The dichotomy in item (1) is not restrictive as if $\N_+=B \sqcup D$ and $\emptyset\neq D\in\I$, then $B\subseteq_\I\N_+$, and analogously, $D\subseteq_\I\N_+$ when $B\in\I$.

\begin{definition}[{\cite[Definition 3.1]{Ghosh}}]\label{Isp} Let $\uu\in\A$ and let $\I$ be a free ideal of $\N$. 
Then $\uu$ has the {\em $\I$-splitting property} if there exists a partition $\N_+=B \sqcup D$ such that: 
\begin{itemize}
      \item[(a)] $B$ and $D$ are either empty or $B, D\not\in\I$;
      \item[(b)] if $B\not\in\I$, then $B$ is $b$-bounded mod $\I$;
      \item[(c)] if $D\not\in\I$, then $D$ is $b$-divergent mod $\I$. %there exists $D' \subseteq \N$ with $D =^{\I} D'$ such that $D'$ is $b$-divergent. 
\end{itemize}
\end{definition}

The dichotomy in item (a) is not restrictive as if $\N_+=B \sqcup D$ and $\emptyset\neq D\in\I$, then $B\subseteq_\I\N_+$, and analogously, $D\subseteq_\I\N_+$ when $B\in\I$.
More on these (and other) ideas can be found in \S \S \ref{Buu}, \ref{Splitsec} (the latter contains also diagrams relating 
$\B_\uu, \D_\uu$ and other ideals). 

The following is the counterpart of Corollary~\ref{splitting} for topologically $\uu_\I$-torsion elements. It is a consequence of Theorem~\ref{conjecture} in the case when $\uu$ has the $\I$-splitting property. Moreover, an application of Corollary~\ref{ThGh:May29} to the case $\I=\F in$ gives Corollary~\ref{splitting}.

\begin{corollary}\label{ThGh:May29}  Let $\uu\in\mathcal A$, let $\I$ a translation invariant free $P$-ideal of $\N$ and $x \in[0,1)$ with $S=\supp(x)$ and $S_b=\supp_b(x)$. If $\mathbf u$ has the $\I$-splitting property  witnessed by the partition $\N_+=B\sqcup D$, then $\bar x \in t^{\I}_\uu(\T)$ if and only if the following conditions hold:
\begin{itemize}
   \item[$\unox$] $(B\cap S)+1 \subseteq^{\I} S$, $B\cap S \subseteq^{\I} S_b$ and if $B\cap S \not\in \I$ then there exists $C\subseteq_{\I} B\cap S$ such that $\lim_{n \in C}\frac{c_{n+1}+1}{b_{n+1}}=1$; 
   \item[$\duex$] if $B\setminus S\not\in\I$, then there exists  $C \subseteq_{\I} B\setminus S$  such that $\lim_{n \in C}\frac{c_{n+1}}{b_{n+1}}=0$;
   \item[$\trex$] if $D\cap S\not\in\I$, then there exists $E \subseteq_{\I} D\cap S$ such that $\lim_{n \in E} \varphi\left(\frac{c_n}{b_n}\right)=0$.
\end{itemize}
\end{corollary}

\begin{remark}\label{Rem:June14}
(a) $\unox$ is vacuously satisfied when $B\cap S \in \I$. 

(b) Corollary \ref {ThGh:May29}, with slight differences in the three conditions (which anyway are equivalent to ours), was stated in \cite[Theorem 3.6]{Ghosh} but with no proof; in fact, the author says: ``The proof follows from similar line of arguments as in~\cite[Corollary 3.13]{DI} with suitable modifications and so is omitted.'' 
\end{remark}

We conclude by recalling two results from~\cite{Nested}  concerning two special cases of Theorem~\ref{conjecture}, namely, when $\supp(x)$ is $b$-divergent or $b$-bounded. The following result, extending~\cite[Corollary~3.4]{DI} and~\cite[Corollary 3.10]{DG1}, covers the case when $\supp(x)$ is $b$-divergent. The sufficiency of the conditions $\Ix$ and $\IIx$ for $\bar x\in t_\uu^\I(\T)$ is directly proved, without any recourse to the sufficiency of the conditions in Theorem~\ref{conjecture}, 

\begin{theorem}[\cite{Nested}]\label{Last:corollary} 
Let $\uu\in\A$, let $\I$ be a translation invariant free $P$-ideal of $\N$ and $x\in[0,1)$ with $S=\supp(x)$ $b$-divergent mod $\I$. Then $\bar x\in t_\uu^\I(\T)$ if and only if  either $S\in\I$ or, in case $S\not\in\I$, the following conditions hold:
\begin{itemize}
    \item[$\Ix$] there exists $D\subseteq_\I S$ such that $\lim_{n\in D}\varphi\left(\frac{c_n}{b_n}\right)=0$;
    \item[$\IIx$] for every $D\in\P(\N)\setminus\I$ with $D\subseteq^\I S$ such that $D-1\in \B_\uu$, there exists $D'\subseteq_\I D$ such that $\lim_{n\in D'}\frac{c_n}{b_n}=0$.
\end{itemize}
\end{theorem}

Similarly, when $\supp(x)$ is $b$-bounded, the following result from~\cite{Nested} extends~\cite[Corollary 3.9]{DG1} and corrects~\cite[Corollary 2.12]{Ghosh}.

\begin{theorem}[\cite{Nested}]\label{Nuovo:Th}%{\rm[$\bar x\in t_\uu^\I(\T)\Leftrightarrow\ax \Leftrightarrow \Tx'$]} 
Let $\uu\in\A$, $\I$ a translation invariant free ideal of $\N$ and $x\in[0,1)$. If $\supp(x)$ is $b$-bounded mod $\I$, then $\bar x\in t_\uu^\I(\T)$ if and only if $\axdue$ as well as the following two conditions hold: 
\begin{itemize}
  \item[$\ix$] $\supp(x)+1\subseteq^\I \supp(x)$; %$(\supp(x)+1) \setminus \supp(x) \in \I$; 
  \item[$\iix$] $\supp_b(x)\subseteq_\I \supp(x)$.
\end{itemize} 
\end{theorem}

For some more recent results on $\I$-characterized subgroups of $\T$ see~\cite{Medit2024,BLMS2025,Expo2025,BSM2025,Leo}.

\smallskip
We  end the introduction by briefly describing the content of the paper.
In \S\ref{Prel} we recall properties of ideals and ideal convergence (\S\ref{Idealsec}), as well as important results used in the main part of the paper (\S\ref{Sec:Approx}) and discuss the ideals $\B_\uu$ and $\D_\uu$ (\S\ref{Buu}). \S\ref{Mainsec} contains the proof of Theorem~\ref{conjecture}. We prove separately the necessity (\S\ref{Necsec}) and the sufficiency (\S\ref{Suffsec}). In \S\ref{Splitsec} Theorem~\ref{conjecture} is analyzed in the special case when the sequence $\mathbf u$ has the $\I$-splitting property. Finally, in \S\ref{notinTsec} we find sufficient conditions for an element of $\T$ to be not topologically $\uu_\I$-torsion.

In the Appendix \S\ref{appendix} we provide an alternative proof of the sufficiency part of Theorem \ref{conjecture} following the strategy from \cite{DiD,DI,Ghosh}, namely Theorem~\ref{SuffiDL} (see Remark \ref{implies*} for more detail). However, this can be achieved only by imposing the additional strong condition (DL) on the ideal $\I$ (see  Definition \ref{Def:DL}). Unfortunately, the condition (DL) is so strong that even $\I_d$ fails to satisfy it (see Example \ref{Laaaaast:Example}(b)), yet $\F in$ satisfies (DL), so 
Theorem \ref{conjecture} completely covers the proof of Theorem \ref{DiD}.

\medskip
\noindent \textbf{Acknowledgements.} 
It is a pleasure to thank Lydia Au\ss enhofer for pointing out an error in Claim \ref{Claim2.10}.

\subsection*{Notation and terminology}

As usual, denote by $\N$ the set of natural numbers and by $\P(\N)$ its power set. For $A\subseteq \N$, let $A^*:=\N\setminus A$ and for an ideal $\I$ let $\I^*:=\{A^*:A\in\I\}$. If $\I_1$ and $\I_2$ are ideals, then the ideal generated by $\I_1\cup\I_2$, denoted by $\I_1+\I_2$, consists of all unions $I_1 \cup I_2$, with $I_s \in \I_s$ for $s = 1,2$. 

Moreover, let $\N_+=\{n\in\N:n>0\}$, and $\Z$, $\Q$ and $\R$ are respectively the group of integers, rationals and reals.
For $m\in \N_+$, $\Z(m) \cong \Z/m\Z$ denotes the cyclic group of order $m$, and for $p$ a prime let $\Z(p^\infty)$ denote the Pr\"ufer group.

For $r\in\R$, denote by $\{r\}$ its fractional part and for $x,y\in\R$, we write $x\equiv_\Z y$ if and only if $x-y\in\Z$.
The function  $\| -\|\colon\T\to[0,1/2]$, defined by $\|x+\Z\|=\min\{|x-n|:n\in\Z\}$ for every $x\in\R$, is a norm on $\T$.

\section{Preliminary results}\label{Prel}

\subsection{Ideals and ideal convergence}\label{Idealsec}

 In the following lemma we collect several basic properties that we use in all the paper.

\begin{lemma}\label{New:Lemma}
Let $\I$ be an ideal of $\N$ and $A,B, C, A_1,\ldots, A_k\in\P(\N)$.
\begin{itemize}
  \item[(a)] If $A_i \subseteq_\I B$ $($resp., $A_i \subseteq^\I B)$ for $i\in\{1,2,\ldots, k\}$, then $\bigcap_{i=1}^k A_i \subseteq_\I B$ $($resp., $\bigcup_{i=1}^k A_i \subseteq^\I B)$. 
  \item[(b)] $A \subseteq^\I B$ if and only if $A \cap B \subseteq_\I A$. 
  \item[(c)]  If $A \subseteq^\I B$ $($resp., $A \subseteq_\I B)$ and $\I$ is translation invariant, then $A + 1 \subseteq^\I B+1$ $($resp., $A + 1 \subseteq_\I B+1)$. 
  \item[(d)] If $A \subseteq^\I B$ $($resp., $A \subseteq_\I B)$ and $B \subseteq^\I C$ $($resp., $B \subseteq_\I C)$,  then $A \subseteq^\I C$ $($resp., $A \subseteq_\I C)$.
  \item[(e)] If $A\cap B \in \I$ and  $C \subseteq^\I B$, then  $A\cap C \in \I$. 
\end{itemize}
\end{lemma}

We will also need the following equivalent condition for an ideal to be translation invariant. The proof is immediate.

\begin{lemma}\label{ti1}
An ideal $\I$ of $\N$ is translation invariant if and only if $(A-1)\cap \N\in\I$ and $A+1\in\I$ for every $A\in\I$. 
\end{lemma}

Next we recall useful notions from~\cite{BCsurvey}. Every infinite, non-cofinite subset $A$ of $\N$ can be written as 
$A=\bigcup_{n\in\N}[k_n,l_n]$, where $k_n,l_n\in\N$ and $k_{n+1}>l_n+1$.
Here $A_n= [k_n,l_n]$ is the \emph{$n$-th maximal interval} of $A$ and $G_n= (l_n,k_{n+1})$ the \emph{$n$-th gap} of $A$. For every $n\in\N$, let
$$c_n(A):=|[k_n,l_n]|=l_n-k_n+1\quad\text{and}\quad g_n(A):=|(l_n,k_{n+1})|=k_{n+1}-l_n-1;$$ moreover, let $g_{-1}(A):=|[0,k_0)|=k_0$. Then let 
$$c(A):=\sup_{n\in\N} c_n(A) \quad\text{and}\quad g(A):=\sup_{n\in\N\cup\{-1\}} g_n(A).$$ 
For $A\in\P(\N)$, $c(A)$ is the \emph{convexity number} and $g(A)$ is the \emph{gap number} of $A$. Clearly, $c(A)=g(A^*)$ and, for $A,B$ infinite non-cofinite subsets of $\N$, if $A \subseteq B$, then $c(A)\leq c(B)$ and $g(A) \geq g(B)$.

The above definition of $c(A)$ and $g(A)$ makes sense for $A$ infinite and non-cofinite. We extend it to all $A\in\mathcal P(\N)$ as follows: if $A$ is finite, then $A$ has finitely many maximal intervals and only one of its gaps is infinite, so we let $c(A):=\max_n c_n(A)<\infty$ (with $c(\emptyset):=0$) and $g(A):=\infty$; dually, if $A$ is cofinite, then $g(A):=\max_n g_n(A)<\infty$ (with $g(\N):=0$) and $c(A):=\infty$. 
With this convention, the equality $c(A)=g(A^*)$ and the monotonicity of $c$ and $g$ hold for all $A,B\in\mathcal P(\N)$: if $A\subseteq B$, then $c(A)\le c(B)$ and $g(A)\ge g(B)$. Consequently, for every $A\in\mathcal P(\N)$,
$$c(A)<\infty \implies A \text{ is not cofinite} \quad\text{and}\quad g(A)<\infty \implies A \text{ is infinite}.$$

\begin{lemma}\label{ginfty}
Let $\I$ be a proper free ideal of $\N$ and $X\subseteq \N$. 
\begin{itemize}
    \item[(a)] If $X\in\I$ and $X$ is $\I$-translation invariant, then $g(X)=\infty$.
    \item[(b)] If $X^*\in\I$ and $X^*$ is $\I$-translation invariant, then $c(X)=\infty$.
\end{itemize}
\end{lemma}
\begin{proof}
(a) If $g(X)=\ell<\infty$, then $\bigcup_{i=0}^\ell(X-i)\supseteq \N$. On the other hand, $\N \cap (X-i)\in\I$ for every $i\in\{0,\ldots,\ell\}$ by hypothesis, and so $\N\in\I$, which means that $\I$ is not proper.

(b) By (a), $g(X^*)=\infty$, that is, $c(X)=\infty$.
\end{proof}

We start recalling the following notion due to Cartan in full generality.

\begin{definition}[Cartan~\cite{Cartan}]\label{Def_:I-conv}
Let $\I$ be an ideal of $\N$. In a topological space $X$, a sequence $(x_n)$ is:
\begin{itemize}
   \item[(a)] \emph{$\I$-convergent} to $x\in X$ if for every neighborhood $U$ of $x$ in $X$, $\{n\in\N: x_n\not\in U\}\in\I$, and we write $x_n\overset{\I}\longrightarrow x$. 
   \item[(b)] \emph{$\I^*$-convergent} to $x\in X$ if there exists $M \in \I^*$ such that $x_n\underset{n\in M}\longrightarrow x$, and we write $x_n\overset{\I^*}\longrightarrow x$. 
\end{itemize}
\end{definition}

Clearly, if $\I\subseteq\mathcal J$ are ideals of $\N$, then $x_n\overset{\I}\longrightarrow x$ implies $x_n\overset{\mathcal J}\longrightarrow x$.

\smallskip
Both $\mathcal Fin$-convergence and $\mathcal Fin^*$-convergence coincide with the usual convergence. More in general, for $P$-ideals and first countable topological spaces, $\I$-convergence coincides with $\I^*$-convergence:
	
\begin{lemma}[{\cite[Proposition 3.2 and Theorem 3.2(ii)]{KSW}}]\label{I^*1} 
Let $\I$ be a free $P$-ideal of $\N$, $X$ a first countable topological space, $x\in X$ and $(x_n)$ a sequence in $X$. Then $x_n\overset{\I}\longrightarrow x$ if and only if $x_n\overset{\I^*}\longrightarrow x$.
\end{lemma}	

Next, we prove the following statement given in \cite[Lemma 2.6]{Ghosh} without a proof. This is fundamental in the proof of the sufficiency of the conditions $\ax$ and $\bx$ to get $\bar x\in t_\uu^\I(\T)$ in the main theorem.

\begin{lemma}\label{Lemma2.6} 
Let $X$ be a first countable topological space, $\I$ a free $P$-ideal of $\N$ and $(x_n)$ a sequence in $X$.
Then $x_n \overset{\I}{\rightarrow} x$ if and only if for any $A\in\P(\N)\setminus \I$ there exists an infinite $A' \subseteq A$ such that $\lim_{n \in A'}x_n=x.$
\end{lemma}
\begin{proof} 
Assume that $x_n \overset{\I}{\rightarrow} x$. As $\I$ is a $P$-ideal, there exists $M\in\I^*$ such that $\lim_{n\in M} x_n = x$, by Lemma~\ref{I^*1}. 
Let $A\in\P(\N)\setminus \I$. Then  $A'=A\cap M\subseteq A$ is obviously infinite and  $\lim_{n\in A'} x_n = x$.

Now assume that $x_n$ does not $\I$-converge to $x$. Then there exists a neighborhood $U$ of $x$ in $X$ such that 
$A=\{n\in\N:x_n\not\in U\}\not\in\I$; hence, $A$ is infinite and $x_n \underset{n\in A'} \nrightarrow x$ for every infinite $A'\subseteq A$. 
\end{proof}

\subsection{Estimations based on the representation with respect to $\uu\in\A$}\label{Sec:Approx}

Here we follow \cite{DiD,DI} for the next needed estimates. For $\uu\in\mathcal A$, $x\in[0,1)$ written as in~\eqref{ex-4}, $n\in\N_+$ and $k\in\N$, let 
$$\sigma_{n,k}(x):=\frac{c_n}{b_n}+\frac{c_{n+1}}{b_n b_{n+1}}+ \cdots + \frac{c_{n+k}}{b_nb_{n+1} \cdots b_{n+k}}.$$

\begin{lemma}[{\cite[Claim 2.7, Lemma 2.5]{DI}}]\label{ug10} 
Let $\uu\in\mathcal A$ and $x\in[0,1)$ written as in~\eqref{ex-4}.
\begin{itemize}
\item[(a)] For every $n \in \N_+$ and for every $k\in \N$,
\begin{equation} \label{eqn:ug4}
\{u_{n-1}x \}=\sigma_{n,k}(x) + \frac{\{u_{n+k}x\}}{b_nb_{n+1} \cdots b_{n+k}}.
\end{equation}
If $k=0$ in~\eqref{eqn:ug4}, then:
\begin{equation} \label{eqn:ug6}
\{u_{n-1}x \}=\frac{c_n}{b_n}+\frac{\{u_nx\}}{b_n}.
\end{equation}
\item[(b)] If $n\in\N_+$ and $k\in\N$ are such that $n,n+1, \ldots, n+k \in \supp_b(x)$, then
\begin{equation}\label{eqn:ug10}
\sigma_{n,k}(x)=1-\frac{1}{b_n b_{n+1} \cdots b_{n+k}}\geq 1-\frac{1}{2^{k+1}}.
\end{equation}
  \item[(c)] If $n\in\N_+$ and $k\in\N$ are such that $n,n+1, \ldots, n+k \not \in \supp(x)$, then $\sigma_{n,k}(x) =0$.
  \item[(d)] For every $n \in \N_+$ and $k\in\N$,
\begin{equation}\label{eqn:ug8}
\begin{split}
\{u_{n-1}x \}&=\sigma_{n,k}(x)+\frac{c_{n+k+1}}{b_n b_{n+1} \cdots b_{n+k+1}}+\frac{\{u_{n+k+1}x  \}}{b_n b_{n+1} \cdots b_{n+k+1}}
\end{split}\end{equation}
and
\begin{equation} \label{eqn:ug9}
\sigma_{n,k}(x) \leq \{u_{n-1}x  \} < \sigma_{n,k}(x)+\frac{c_{n+k+1}}{b_n b_{n+1} \cdots b_{n+k+1}}+\frac{1}{2^{k+2}}.
\end{equation}
\end{itemize}
\end{lemma}

The formula \eqref{eqn:ug8} and the estimate \eqref{eqn:ug9} are precisely~\cite[Eq.~(10), Eq.~(11)]{DI}, respectively.

\medskip
The next lemma deals with the case when $\supp(x)$ belongs to the given ideal. 
The implication $\supp(x)\in\I\Rightarrow \bar x\in t_\uu^\I(\T)$ is proved similarly in \cite[Lemma 2.2]{Ghosh} under more stringent assumptions. 
%but here we avoid the assumptions on the ideal to be an analytic $P$-ideal; the proof is almost the same, we repeat it here to keep the paper self-contained.
%is a particular case of the same result for non-snt analytic $P$-ideals in \cite{DasGhosh}. %for a proof of the following even more general case see DGEggleston.

\begin{lemma}\label{Lemma2.2}%[{\cite[Lemma 3.2]{DasGhosh}} and \cite{Ghosh1}]\label{3.2}
Let $\uu\in\A$, let $\I$ be a proper free ideal of $\N$ and $x\in[0,1)$. If $\supp(x)\in\I$ and $\supp(x)$ is $\I$-translation invariant, then $g(\supp(x))=\infty$ and $\bar x\in t_\uu^\I(\T)$.
\end{lemma}
\begin{proof}
Let $S=\supp(x)$.  That $g(S)=\infty$ follows from Lemma~\ref{ginfty}. 

For $k \in \N_+$, define $S_k=\N \cap \bigcup_{i=0}^{k}(S-i)$. As $S$ is $\I$-translation invariant, $\N \cap(S-i) \in \I$ for every $i\in\{0, \ldots,k\}$ and therefore $S_k \in \I$. 
Now we verify that $\bar x\in t_\uu^\I(\T)$. Fix $k\in\N_+$.
Observe that $n \in \N \setminus S_k$ precisely when $n+i \notin S$ for every $i\in\{0, \ldots, k\}$, equivalently
$c_n=c_{n+1}= \ldots =c_{n+k}=0$.
Consequently, if $n\in \N \setminus S_k$, then $$\{u_nx \}=\sum_{i=n+k+1}^{\infty} \frac{c_i}{u_i}\cdot u_n \leq u_n \cdot  \sum_{i=n+k+1}^{\infty} \frac{b_i-1}{u_i}
=u_n \cdot  \sum_{i=n+k+1}^{\infty}  \left(\frac{1}{u_{i-1}}-\frac{1}{u_i}  \right) \leq \frac{u_n}{u_{n+k}}  \leq \frac{1}{2^k}.$$
This proves that $\Vert u_n x\Vert<\frac{1}{2^k}$ for every $n \in \N \setminus S_k$. Therefore, $\{n\in\N:\Vert u_n x\Vert \geq \frac{1}{2^k}\}\subseteq S_k\in\I$, and so $\bar x\in t_\uu^\I(\T)$.
\end{proof}

\subsection{The ideals $\B_\uu$ and $\D_\uu$}\label{Buu}

Next we clarify the role of the ideals $\B_\uu$ and $\D_\uu$.
We introduce other two ideals, namely,  $\I^{bd} :=  \B_\uu + \I$ and $\I^{div} :=  \D_\uu + \I$, that contain all $b$-bounded mod $\I$ and all $b$-divergent mod $\I$ subsets, respectively:

\begin{lemma}\label{New:claim}
For $\uu\in\mathcal A$, a free ideal $\I$ of $\N$ and $X\in P(\N)$, $X \in \I^{bd}$ (resp.,  $X\in \I^{div}$) if and only if $X$ is $b$-bounded mod $\I$ (resp., $X$ is $b$-divergent mod $\I$).

In particular, $\uu$ is $b$-bounded mod $\I$ (resp., $b$-divergent mod $\I$) if and only if $\I^{bd} =\P(\N)$ (resp., $\I^{div}=\P(\N)$).
\end{lemma}

The easy proof of the above lemma is left to the reader.

\begin{corollary}
For $\uu\in\A$, 
\begin{itemize}
  \item[(a)] $\uu$ is $b$-bounded $($i.e., $\B_\uu=\P(\N))$ if and only if  $\D_\uu = \F in$.
  \item[(b)] $\uu$ is $b$-divergent $($i.e., $\D_\uu=\P(\N))$ if and only if $\B_\uu = \F in$.
\end{itemize}
\end{corollary}

Since $\I\subseteq \I^{bd}$, the subgroup  $t_\uu^{\I}(\T)$ we are always interested in is contained in $t_\uu^{\I^{bd}}(\T)$. As the subgroup $t_\uu^{\I^{bd}}(\T)$ is easy to compute via Corollary~\ref{IcB}, this can be applied when we are interested to see that some $\bar x \not \in t_\uu^{\I}(\T).$ It is enough (although not necessary) to check that $\bar x \not \in t_\uu^{\I^{bd}}(\T)$. 

The following lemma in particular covers~\cite[Lemma 2.8]{Ghosh}. The inclusion $\B_\uu \cap \P(A) \subseteq \I$ for $A\in\P(\N)\setminus\I$ obviously means that all $b$-bounded subsets of $A$ belongs to $\I$.

\begin{lemma}\label{Lemma2.8*} 
Let $\uu\in\mathcal A$, let $\I$ be a free $P$-ideal of $\N$ and let $A \in\P(\N)\setminus \I$. If $\B_\uu \cap \P(A) \subseteq \I$, then $A\in\I^{div}$. 
\end{lemma}
\begin{proof} 
For $k\in\N_+$ the set $A_k:=\{n\in A:b_n\leq k\}$ is $b$-bounded and contained in $A$, hence $A_k\in\I$ by assumption. Since $\I$ is a $P$-ideal, there exists $C\in\I$ such that $|A_k\setminus C|<\infty$ for every $k\in\N$. Then $B:=A\setminus C\subseteq_\I A$. For every $k\in\N_+$, since $|A_k\cap B|=|A_k\setminus C|<\infty$, there exists $n_k\in\N$ such that $A_k\cap B\subseteq [0,n_k]$. Now, if $n\in B\subseteq A$ and $n>n_k$, then $n\notin A_k$, hence $b_n>k$. We have seen that for all $k\in\N_+$ there exists $n_k\in\N$ such that $b_n>k$ for all $n\in B$ with $n>n_k$, i.e., $B$ is $b$-divergent.
\end{proof}

\section{The main theorem}\label{Mainsec}

In this section we present a proof of our main theorem, namely, Theorem~\ref{conjecture}.

\subsection{Two results from \cite{DG1} and~\cite{Ghosh}}\label{Two:indians}

In the sequel we recall two recent results from \cite{DG1} and~\cite{Ghosh} similar to Theorem~\ref{conjecture} and compare them with Theorem~\ref{conjecture}. 

 As in  \cite{DG1}, for $A, B \subseteq \N$, we write $A \subseteq^{\alpha } B$, whenever $d_{\alpha} (A \setminus B)=0$, and $A\subseteq_\alpha B$, whenever $A\subseteq B$ and $B\subseteq^\alpha A$. 

\begin{theorem}[{\cite[Theorem 2.1]{DG1}}]\label{Th:DG} Let $\uu\in\mathcal A$, $\alpha\in(0,1]$ and $x \in [0,1)$. Then $\bar x\in t^{\alpha}_{\uu}(\T)$ if and only if either $d_{\alpha}(\supp(x))=0$ or if $d_{\alpha}(\supp(x))>0$, then for all $A \subseteq \N$ with $d_{\alpha}(A)>0$ the following holds.
\begin{itemize}
  \item[$\mathrm{(a)}$] If $A$ is $b$-bounded then:
\begin{itemize}
  \item[$\mathrm{(a1)}$] if $A \subseteq^{\alpha } \supp(x)$, then $A+1\subseteq^{\alpha } \supp(x)$, $A \subseteq^{\alpha } \supp_b(x)$ and there exists $A' \subseteq_\alpha A$ such that $\lim_{n \in A'} \frac{c_{n+1}+1}{b_{n+1}}=1$. 
\\Moreover, if $A+1$ is $b$-bounded, then $A+1 \subseteq^{\alpha }\supp_b(x)$.
  \item[$\mathrm{(a2)}$] If $d_{\alpha}(A\cap \supp(x))=0$, then there exists $A' \subseteq_\alpha A$ such that $\lim_{n \in A'} \frac{c_{n+1}}{b_{n+1}}=0$.
\\ Moreover,  if $A+1$ is $b$-bounded, then $d_{\alpha}((A+1)\cap \supp(x))=0$ as well.
\end{itemize}
\item[$\mathrm{(b)}$] If $A$ is $b$-divergent, then there exists $B\subseteq_\alpha A$ such that $\lim_{n \in B}\varphi\left(\frac{c_n}{b_n}\right)=0$.
\end{itemize}
\end{theorem}

This theorem obviously follows from the next one with $\I=\I_\alpha$: 

\begin{theorem}[{\cite[Theorem~2.9]{Ghosh}}]\label{Theorem2.9} 
Let $\uu\in\mathcal A$, let $\I$ be a translation invariant free $P$-ideal of $\N$ and let $x \in[0,1)$. Then $\bar x\in t^{\I}_\uu(\T)$ if and only if either $\supp(x) \in \I$ or if $\supp(x) \notin \I$, then for all $A \in\P(\N)\setminus\I$ the following holds.
\begin{itemize}
   \item[$\mathrm{(a)}$] If $A\in \B_\uu$,  then:
\begin{itemize}
     \item[$\mathrm{(a1)}$] if $A \subseteq^{\I} \supp(x)$, then $A \subseteq^{\I} \supp_b(x)$, $A+1\subseteq^{\I} \supp(x)$ and there exists $A' \subseteq_{\I} A$ such that 
      $\lim_{n \in A'} \frac{c_{n+1}+1}{b_{n+1}}=1$. 
     Moreover, if $A+1$ is $b$-bounded, then $A+1\subseteq^\I \supp_b(x)$.
     \item[$\mathrm{(a2)}$] if $A\cap \supp(x) \in \I$, then there exists $B' \subseteq_{\I} A$ such that $\lim_{n \in B'} \frac{c_{n+1}}{b_{n+1}}=0$.
     Moreover,  if $A+1$ is $b$-bounded, then $(A+1)\cap \supp(x) \in \I$ as well.
\end{itemize}
\item[$\mathrm{(b)}$] If $A\in\D_\uu$, then there exists $B \subseteq_{\I} A$ such that $\lim_{n \in B} \varphi\left(\frac{c_n}{b_n}\right)=0$.
\end{itemize}
\end{theorem} 

\begin{remark}\label{implies*}
(a) The proof, given in \cite{Ghosh}, of the sufficiency of the conditions (a) and (b) in Theorem~\ref{Theorem2.9} presents a flaw.  It follows step by step the proof given in~\cite{DG1} of Theorem~\ref{Th:DG}. The latter one is
based on the strategy adopted in \cite{DI} for Theorem~\ref{DiD} and $\I = \F in$. Unfortunately, the proofs in \cite{DG1,Ghosh} present a gap. This gap is precisely in Subcase (i$_b$) of the proof of the sufficiency of~\cite[Theorem~2.1]{DG1}, where the authors claim that ``If this process does not terminate after finitely many steps then we can conclude that there exists $B\subseteq A$ with $\overline d_\alpha(B)>0$ such that $S_k(B)$ is $q$-bounded for all $k\in\N$''. Unfortunately, there is nothing that ensures the existence of such a set $B$.  For more detail see also (i2) in the proof of Theorem~\ref{SuffiDL} in the Appendix \S\ref{appendix}, where a remedy to this problem is proposed by imposing an additional strong condition on the ideal $\I$. 

(b) The proof of the sufficiency of our Theorem~\ref{conjecture}, given in \S\ref{Suffsec}, exploits a completely new idea and does not follow or use the proof of the case $\I = \F in$, namely, of Theorem~\ref{DiD}. Theorem~\ref{conjecture} implies all Theorems~\ref{DiD},~\ref{Th:DG},~\ref{Theorem2.9} (see Remark~\ref{implies}).
\end{remark}

\begin{remark}\label{implies}\label{SinI}
Here we see that Theorem~\ref{Theorem2.9} (so also Theorem~\ref{DiD} and Theorem~\ref{Th:DG}) follows from Theorem~\ref{conjecture}.
To this end we carefully analyze and compare the hypotheses of these theorems. 

In the sequel $\uu\in\mathcal A$,  $\I$ is a translation invariant free ideal of $\N$, $x \in[0,1)$, $S=\supp(x)$ and $S_b=\supp_b(x)$.

(i) First we check that in both items (a1) and (a2) of  Theorem~\ref{Theorem2.9}, the last assertion (starting with ``Moreover'')
immediately follows from the first one, in case $S \not \in \I$. Indeed, for (a1) suppose that $A \subseteq^{\I} S$ and $A+1$ is $b$-bounded, then it suffices to  apply 
 the first assertion to $A+1$. It is applicable as $A+1\subseteq^{\I} S$ and consequently $A+1 \not \in \I$ due to the restraint $S\not \in\I$ explicitly imposed in this theorem
 (if $S\in \I$, then (a1) holds true vacuously). This shows in particular, that $\axuno$ implies (a1).
 
 As far as (a2) is concerned, if $A\cap S\in\I$ and $A+1$ is $b$-bounded, then  the first assertion of (a2) implies the existence of $B' \subseteq_{\I} A$ such that $\lim_{n \in B'} \frac{c_{n+1}}{b_{n+1}}=0$. Since $B'+1\subseteq A+1$, the set $B'+1$ is $b$-bounded and so, in view of the above limit, $(B'+1)\cap S \in \F in \subseteq \I$.
Since,   $B'+1\subseteq_\I A+1$, we conclude that $(A+1)\cap S\in\I$. This shows, among others, that $\axdue$ and (a2) are equivalent. 

(ii) The condition $S\not \in\I$ that appears in Theorem~\ref{Theorem2.9} (and Theorem~\ref{Th:DG}, for $\I = \I_\alpha$) is missing in Theorem~\ref{conjecture}.     
Let us see that if $S\in\I$ then all five conditions $\bar x\in t^\I_{\uu}(\T)$, $\ax$, (a) and $\bx = \mathrm{(b)}$ hold simultaneously, hence 
imposing $S\not \in\I$ in  Theorem~\ref{Theorem2.9} (and Theorem~\ref{Th:DG}) is not necessary. 

In fact, the first one follows from Lemma \ref{Lemma2.2}. 
Suppose next that $A\in\P(\N)\setminus\I$.  If $A$ is $b$-bounded,  we have to verify only $\axdue=\mathrm{(a2)}$, since $A\not \subseteq ^\I S$ (in view of $S\in\I$),
so $\axuno$ and (a1) are vacuously satisfied. To check $\axdue$, note that $A\cap S\in\I$.
The set $B'=A\setminus(S-1)$ satisfies $B'\subseteq_\I A$, as $\I$ is translation invariant, so $S-1\in\I$.
Clearly, $c_{n+1}=0$ for every $n\in B'$, and so $\lim_{n\in B'}\frac{c_{n+1}}{b_{n+1}}=0$. This verifies $\axdue$ and (a2). 

If $A$ is $b$-divergent, then $A\setminus S\subseteq_\I A$, since $A\cap S\in\I$, and $\lim_{n \in A\setminus S} \varphi\left(\frac{c_n}{b_n}\right)=0$, as $c_n=0$ for every $n\in A\setminus S$.

(iii) To see that Theorem~\ref{Theorem2.9} follows from Theorem~\ref{conjecture} assume first that 
$\uu$, $\I$ and $x$ satisfy (a) and (b) of Theorem~\ref{Theorem2.9}. To prove that $\bar x\in t^\I_{\uu}(\T)$ it suffices
to check $\ax$ and $\bx$ from Theorem~\ref{conjecture}. As mentioned in (i), $\axdue$ and (a2), as well as $\bx$ and (b) are equivalent. 
Hence, it remains to see that (a1) of Theorem~\ref{Theorem2.9} implies $\axuno$. To this end take an $A'$ as in (a1), then to obtain the only missing property of $A'$ in $\axuno$ just replace $A'$ with $A''=A' \cap S_b$. Then $A' \subseteq_{\I} A$ and $A \subseteq^{\I}S_b$, give the chain $A'' = A' \cap S_b \subseteq_{\I} A\cap S_b \subseteq_{\I} A$,  so $A'' \subseteq_{\I} A$ and obviously $A'' \subseteq S_b$. 

Finally, to see that the necessity of (a) and (b) of Theorem~\ref{Theorem2.9} follows from  that of $\ax$ and $\bx$ of Theorem~\ref{conjecture} assume that $\bar x\in t^\I_{\uu}(\T)$.
Then $\ax$ and $\bx$ hold in view of Theorem~\ref{conjecture}. As we saw above,  $\ax$ implies (a1) of Theorem~\ref{Theorem2.9},
while $\axdue$ and (a2), as well as $\bx$ and (b) are equivalent. Hence, (a) and (b) hold. 
%%%%%%%%
%The second (``Moreover'') part of Theorem~\ref{Theorem2.9}(a1) follows from the first part, which, in turn
%is equivalent to $\axuno$. Indeed, if $A \subseteq^{\I} \supp(x)$ and $A+1$ is $b$-bounded, then $\axuno$ applied to $A$ gives in particular $A+1\subseteq^{\I} \supp(x)$, so $\axuno$ can be applied also to $A+1$ giving that $A+1\subseteq^\I \supp_b(x)$. 
%
%
%Analogously, the second (``Moreover'') part of Theorem~\ref{Theorem2.9}(a2) follows from the first statement, which coincides with $\axdue$. Indeed, assume that 
%$A\cap \supp(x)\in\I$ and that $A+1$ is $b$-bounded. By $\axdue$,  there exists $B' \subseteq_{\I} A$ such that $\lim_{n \in B'} \frac{c_{n+1}}{b_{n+1}}=0$.
%Since $B'+1\subseteq A+1$, $B'+1$ is $b$-bounded and so $(B'+1)\cap \supp(x)$ is finite; as $B'+1\subseteq_\I A+1$, we conclude that $(A+1)\cap\supp(x)\in\I$.
\end{remark}

\subsection{Necessity}\label{Necsec}

Following \cite{Ghosh}, for $\uu\in\mathcal A$ and an infinite $B\subseteq \N$, let $t_{\uu_{B}}(\T):= \left\{x\in \T: \lim_{n\in B} u_nx =0\right\}$. 

 The following lemma is used in the proof of the necessity; it is the counterpart of \cite[Lemma~2.6]{DI}, more precisely, \cite[Lemma~2.6]{DI} is Lemma~\ref{Lemma2.7} with $\I=\F in$.
 One can find it in {\cite[Lemma 2.7]{Ghosh}} with more stringent hypotheses on the ideal.

\begin{lemma}\label{Lemma2.7}
Let $\uu\in\mathcal A$, let $\I$ be a free ideal of $\N$ and let $x\in[0,1)$ with $S=\supp(x)$ and $S_b=\supp_b(x)$.
Let $B\in\P(\N)\setminus \I$ and assume that  $\bar x \in t_{\uu_{B-1}}(\T)$. 
\begin{itemize}
\item[(a)] If $B \subseteq^{\I} S$ and $B\in \B_\uu$, then  
\begin{equation}\label{dag}
B\cap S \subseteq_{\I} B, \ \lim_{n \in B\cap S}\{u_{n-1}x \}=1\quad\text{and}\quad B\cap S \subseteq^* S_b.
\end{equation}
Consequently, $B \subseteq^{\I} S_b$. 
\item[(b)] If $B\cap S\in \I$, then $B\setminus S \subseteq_{\I} B$ and $\lim_{n \in B\setminus S}\{u_{n-1}x \}=0$.
\end{itemize}
\end{lemma}
\begin{proof} 
(a) Since $B\cap S\subseteq B$, our hypothesis $B\setminus (B\cap S)=B\setminus S \in \I$ implies the inclusion $B\cap S \subseteq_{\I} B$ in \eqref{dag}. 
Let $b=1+\max_{n \in B} \{b_n\}$. By~\eqref{eqn:ug6}, for every  positive $n\in B\cap S$,
\begin{equation}\label{184}
\{u_{n-1}x \}=\frac{c_n}{b_n}+\frac{\{u_nx\}}{b_n}\geq \frac{c_n}{b_n}>\frac{1}{b}
\end{equation}
as each $c_n \geq 1$. Since $\bar x \in t_{\uu_{B-1}}(\T)$,~\eqref{184} implies that
$\lim_{n \in B\cap S}\{u_{n-1}x \}=1.$
Then, using again~\eqref{eqn:ug6}, we conclude that for almost every $n \in B\cap S$ with $n\neq 0$,
$$1-\frac{1}{b_n}<1-\frac{1}{b}<\{u_{n-1}x \}=\frac{c_n}{b_n}+\frac{\{u_nx\}}{b_n}<\frac{c_n}{b_n}+\frac{1}{b_n}.$$
Therefore, for almost every $n \in B\cap S$, $c_n>b_n-2$, and thus $c_n=b_n-1$. 
Hence, $B\cap S \subseteq^* S_b$. This implies that $B \subseteq^{\I}S_b$ as $\I$ is free.

\smallskip
(b) As $B\setminus S\subseteq B$ and $B \setminus (B\setminus S)=B\cap S\in \I$ by hypothesis, we get that $B\setminus S\subseteq_{\I} B$. 
Furthermore, by using~\eqref{eqn:ug6}, for every  positive $n\in B\setminus S$,  
$$\{u_{n-1}x \}=\frac{c_n}{b_n}+\frac{\{u_nx\}}{b_n}=\frac{\{u_nx\}}{b_n}<\frac{1}{b_n}\leq\frac{1}{2},$$
since $c_n=0$ and $b_n \geq 2$. As $\bar x \in t_{\uu_{B-1}}(\T)$, this means that $\lim_{n \in B\setminus S}\{u_{n-1}x \}=0$.
\end{proof}

We are now in position to prove the necessity in Theorem~\ref{conjecture}.

\begin{proposition}[Necessity  in~\ref{conjecture}]\label{necessity} Let $\uu\in\mathcal A$, let $\I$ be a translation invariant free $P$-ideal of $\N$ and let $x \in[0,1)$ with $S=\supp(x)$ and $S_b=\supp_b(x)$. If $\bar x\in t^{\I}_\uu(\T)$, then $\ax\&\bx$ holds for every $A\in\P(\N)\setminus\I$.
\end{proposition}
\begin{proof} Since $\I$ is a $P$-ideal, there exists $N \in \I^*$ such that $\lim_{n \in N}\varphi(u_{n}x)=0$ by Lemma~\ref{I^*1}. Then $M=N+1 \in\I^*$, as $\I$ is translation invariant, and hence, 
\begin{equation} \label{eqn:lim0}
\lim_{n \in M}\varphi(u_{n-1}x)=\lim_{n \in N}\varphi(u_{n}x)=0.
\end{equation}
Let $A\in\P(\N)\setminus\I$ and $B=M \cap A$. Then $B \subseteq A$ and $A \setminus B=A \setminus M\subseteq \N\setminus M \in \I$, so also $A \setminus B\in\I$ and hence $B \subseteq_{\I} A$,  in particular $B$ is infinite. Since $B \subseteq M$, from~\eqref{eqn:lim0} we also obtain that $\lim_{n \in B}\varphi(u_{n-1}x)=0$. This 
proves that
\begin{equation}\label{*eq}
\text{for every $A\in\P(\N)\setminus\I$ there exists an infinite $B \subseteq_{\I} A$ such that $\bar x \in t_{\uu_{B-1}}(\T)$.}
\end{equation}
Let $A\in\P(\N)\setminus\I$ and let $B$ be as in \eqref{*eq}.  Hence, $B \not \in \I$. 

\smallskip
$\ax$ Assume that $A\in \B_\uu$. Then also $B\in \B_\uu$, as $B\subseteq A$. 

\smallskip
$\axuno$ Suppose that $A \subseteq^{\I} S$.  Then $B \subseteq^\I S$, by Lemma~\ref{New:Lemma}(d). According to Lemma~\ref{Lemma2.7}(a), $B \subseteq^{\I} S_b$, hence $S_b \not \in \I$. In view of~\eqref{dag}, $A'':=B\cap S$ satisfies 
\begin{equation}\label{above}
A'' \subseteq_{\I} B, \quad \lim_{n \in A''}\{u_{n-1}x  \}=1 \quad \text{and} \quad A'' \subseteq^* S_b.
\end{equation}
So, $A' := A'' \cap S_b$ is a cofinite subset of $A''$. Hence, $A' \subseteq_{\I} B$ still holds in view of the first inclusion in \eqref{above}. As $B \subseteq_{\I} A$, we deduce that $A' \subseteq_{\I} A$ by Lemma~\ref{New:Lemma}(d). Therefore, $A' \subseteq  S_b$ implies $A \subseteq^{\I} S_b$ by Lemma~\ref{New:Lemma}(d).

Our next aim is to prove that 
\begin{equation} \label{eqn:lim2}
\lim_{n \in A'} \frac{c_{n+1}+1}{b_{n+1}}=1.
\end{equation}
Using~\eqref{above}  (in the first equality below), \eqref{eqn:ug6} (in the second one) and $A' \subseteq S_b$ (in the third one), we get  
$$1=\lim_{n \in A'}\{u_{n-1}x  \}=\lim_{n \in A'}\left(\frac{c_n}{b_n}+\frac{\{u_nx \}}{b_n} \right)=\lim_{n \in A'}\left(\frac{b_n-1+\{u_nx \}}{b_n} \right)=\lim_{n \in A'}\left(1-\frac{1-\{u_nx \}}{b_n} \right).$$
This implies that  
$\lim_{n \in A'}\frac{1-\{u_nx \}}{b_n}=0.$
Since $A' \subseteq B$ is $b$-bounded, this gives $\lim_{n \in A'}(1-\{u_nx \})=0$, that is,
\begin{equation} \label{eqn:lim1}
\lim_{n \in A'}\{u_nx \}=1.
\end{equation}
As $c_{n+1} \leq b_{n+1}-1$ for every $n\in\N$, from~\eqref{eqn:ug6},  we get  
$\{u_nx\}=\frac{c_{n+1}}{b_{n+1}}+\frac{\{u_{n+1}x\}}{b_{n+1}}<\frac{c_{n+1}+1}{b_{n+1}} \leq 1.$
Hence, 
\begin{equation}\label{194}
\lim_{n \in A'}\{u_nx \} \leq \lim_{n \in A'} \frac{c_{n+1}+1}{b_{n+1}} \leq 1.
\end{equation}
Now~\eqref{eqn:lim1} and~\eqref{194} imply \eqref{eqn:lim2}. 

In order to prove the remaining inclusion $A+1 \subseteq^{\I} S$, we start noting that, since $b_{n+1}\geq 2$ for every $n\in\N$, \eqref{eqn:lim2} gives $c_{n+1}+1>1$, i.e., $c_{n+1} \neq 0$, for almost every $n \in A'$.  So, $A'+1 \subseteq^* S$, hence $A'+1 \subseteq^\I S$. As noticed above, $A'\subseteq_\I A$, and so $A'+1\subseteq_\I A+1$ by Lemma~\ref{New:Lemma}(c); in particular, $A+1\subseteq^\I A'+1$. Then $A+1\subseteq^\I S$ by Lemma~\ref{New:Lemma}(d).

\smallskip 
$\axdue$ Now suppose that $A\cap S \in \I$. Since $B \subseteq A$, we have $B\cap S \in \I$, too. By Lemma~\ref{Lemma2.7}(b), $B':=B\setminus S\subseteq_{\I} B$ and satisfies $\lim_{n \in B'}\{u_{n-1}x \}=0$, by \eqref{*eq}. 
By Lemma~\ref{New:Lemma}(d), $B' \subseteq_{\I} B$ and $B \subseteq_{\I} A$ imply that $B'\subseteq_{\I} A$.
As $c_n=0$, for every $n\in B'$, \eqref{eqn:ug8} (with $k=0$) gives 
$\{u_{n-1}x \}= \frac{c_{n+1}}{b_n b_{n+1}}+\frac{\{u_{n+1}x \}}{b_n b_{n+1}}$.
Therefore,
$$\lim_{n \in B'}\left(\frac{c_{n+1}}{b_n b_{n+1}}+\frac{\{u_{n+1}x \}}{b_n b_{n+1}}\right)=\lim_{n \in B'}\{u_{n-1}x \}=0.$$
Since, for every $n\in\N$,  $c_n\geq0$, $\{u_{n}x\} \geq 0$ and $b_n>0$, we obtain $\lim_{n \in B'} \frac{c_{n+1}}{b_n b_{n+1}}=0$.
As $B' \subseteq B$ is $b$-bounded, this gives  $\lim_{n \in B'} \frac{c_{n+1}}{b_{n+1}}=0$. 

\smallskip 
$\bx$ Assume that $A\in\D_\uu$, so $B\in\D_\uu$, too. Then $\bar x\in t_{\uu_{B-1}}(\T)$ and~\eqref{eqn:ug6} imply 
$$\lim_{n \in B}\varphi\left(\frac{c_n}{b_n}+\frac{\{u_nx \}}{b_n}  \right)=\lim_{n \in B}\varphi(\{u_{n-1}x\})=\lim_{n \in B}\varphi(u_{n-1}x)=0.$$
Since $\{u_nx \}<1$ and $\lim_{n \in B} b_n=\infty,$ it follows that $\lim_{n \in B}\varphi\left(\frac{c_n}{b_n}\right)=0$, ending up the proof of the necessity.
\end{proof}

\begin{remark}\label{moreover} 
In $\axuno$ it is possible to achieve also the inclusion $A'+1\subseteq S$. To this end it suffices to replace $A'$ by $A''= A' \cap (S-1) \subseteq_\I A$. Indeed,  the inclusion $A''\subseteq A' \subseteq S_b$ and $A''+1\subseteq S$ are obvious. Furthermore,   
$A+1\subseteq^\I S$ implies that $A\subseteq^\I S - 1$, so $A\cap (S-1) \subseteq_\I A$. Hence $A'' \subseteq_\I A$, in view of $A' \subseteq_{\I} A$. Clearly,   $\lim_{n \in A''} \frac{c_{n+1}+1}{b_{n+1}}=1$. From the proof of $\axdue$, one has that $B'\cap S=\emptyset$.
\end{remark}

\subsection{Sufficiency}\label{Suffsec}

Let $\uu\in\A$, let $\I$ be a translation invariant free $P$-ideal of $\N$ and $x\in[0,1)$.
In order to prove the sufficiency in the main theorem, i.e., that $\bar x\in t_\uu^\I(\T)$ under $\ax$ and $\bx$, we use the following condition for $A\in\P(\N)$:
\begin{itemize}
\item[$(\mathrm{U}_A)$] there exists an infinite $B\subseteq A$ such that $\bar x\in t_{\uu_{B-1}}(\T)$. 
\end{itemize}
Obviously, $(\mathrm{U}_B) \Rightarrow (\mathrm{U}_A)$ for every subset $B$ of $A$. 

The relevance of $(\mathrm{U}_A)$ is determined by the following equivalence, which holds in view  of Lemma~\ref{Lemma2.6}.

\begin{lemma}\label{UAholds}
Let $\uu\in\A$, let $\I$ be a translation invariant free $P$-ideal of $\N$ and let $x\in[0,1)$. Then  $\bar x\in t_\uu^\I(\T)$ if and only if $(\mathrm{U}_A)$ holds for all $A\in\P(\N)\setminus\I$. 
\end{lemma}

\begin{remark}\label{case3}
Let $\I$ be a free ideal of $\N$ and let $A\in\P(\N)\setminus\I$. Let $S\subseteq\N$ (below always $S=\supp(x)$ with $\uu\in\A$ and $x\in[0,1)$).
Since $A \not \in \I$, it is not possible to have both $A\cap S\in \I$ and $A\setminus S\in \I$. 
From the remaining three cases: (1) $A\cap S \in \I$, $A\setminus S\not \in \I$; (2) $A\setminus S \in \I$, $A\cap S \not \in \I$ and (3) $A\cap S \not \in \I$ and $A\setminus S \not \in \I$, we consider below in the proof of the sufficiency only (1) and (2), since the case (3) can be easily reduced to both (1) or/and (2) after replacing $A$ by $A\setminus S \not \in \I$ or $A\cap S \not \in \I$.
In particular, one can always assume that either $A\subseteq S$ or $A\cap S=\emptyset$.
\end{remark}

\subsubsection{Chasing for $(\mathrm{U}_A)$ when $A\in\D_\uu\setminus\I$ or $A\in\B_\uu\setminus\I$ and $A \subseteq \supp(x)$}

In our first lemma we verify that for $x \in[0,1)$ one has $\bx \Rightarrow (\mathrm{U}_A)$ for all $A\in\P(\N)\setminus\I$ that are $b$-divergent.  Actually, for such $A$ the following stronger property holds:

\begin{lemma}\label{L2.3}
Let $\uu\in\mathcal A$, let $\I$ be a translation invariant free $P$-ideal of $\N$ and $x \in[0,1)$ satisfying $\bx$. If $A\in\D_\uu \setminus\I$,
then there exists $B\subseteq_\I A$ such that $\lim_{n \in B}\varphi(u_{n-1}x)=0$; in particular $(\mathrm{U}_A)$ holds. 
\end{lemma}
\begin{proof}
By $\bx$,  $\lim_{n \in B'}\varphi\left(\frac{c_n}{b_n}\right)=0$ for some $B' \subseteq_{\I} A$. So, as $\lim_{n \in B'}\frac{\{u_{n}x\}}{b_n}\leq\lim_{n \in B'}\frac{1}{b_n}=0$,~\eqref{eqn:ug6} gives that
\[\lim_{n \in B'}\varphi(u_{n-1}x)=\lim_{n \in B'}\varphi(\{u_{n-1}x\})=\lim_{n \in B'}\varphi \left(\frac{c_n}{b_n}+\frac{\{u_{n}x\}}{b_n} \right)=0.\qedhere\]
\end{proof}

Following \cite{Ghosh}, for $k \in \N$ and a subset $E$ of $\N$ we put 
\begin{equation}\label{Lk}
S_k(E):= \bigcup_{j = 0}^k (E + j).
\end{equation}

\begin{claim}\label{DD}
Let $\uu\in\A$, let $\I$ be a translation invariant free ideal of $\N$ and $x\in[0,1)$ with $S=\supp(x)$ and $S_b=\supp_b(x)$. If an infinite $E\subseteq \N$  satisfies  $\neg (\mathrm{U}_E)$ and $S_k(E) \subseteq S_b$ for some $k\in \N$ then $F+1\in \B_\uu$ for every infinite $F\subseteq E+ k $ satisfying $\lim_{m \in F} \frac{c_{m+1}+1}{b_{m+1}}=1$.
\end{claim}

\begin{proof} Let $F$ be an  infinite subset of $E+ k $ satisfying $\lim_{m \in F} \frac{c_{m+1}+1}{b_{m+1}}=1$.
Assume for a contradiction that $F+1$ is not $b$-bounded. Then there is a strictly increasing sequence $(m_i)$ in $F$ with $b_{m_i+1}\rightarrow +\infty$; thus, $\lim_{i\to\infty}\frac{c_{m_i+1}}{b_{m_i+1}}=1$, since $\lim_{m\in F} \frac{c_{m+1}+1}{b_{m+1}}=1$ by hypothesis. 
Since $F - k\subseteq E$, $n_i:= m_i - k \in E$ for every $i\in\N$. Next we prove that $\lim_{i\to \infty}\Vert u_{n_i-1}x\Vert=0$ which ensures $(\mathrm{U}_E)$, a contradiction.

Fix an arbitrary $i\in \N$ and for the sake of simplicity write $n:=n_i$. Then $n+j\in S_k(E)\subseteq S_b$ for $j\in\{0,\ldots,k\}$. Hence,  one has: 
\begin{align*}
u_{n-1}x&\equiv_\Z u_{n-1}\left(\sum_{j=0}^k\frac{b_{n+j}-1}{u_{n+j}}+\frac{c_{n+k+1}}{u_{n+k+1}}+\sum_{i>n+k+1}\frac{c_i}{u_i}\right)\geq u_{n-1}\left(\frac{1}{u_{n-1}}-\frac{1}{u_{n+k}}\right)+u_{n-1} \frac{c_{n+k+1}}{u_{n+k+1}}\\
&=1-\frac{u_{n-1}}{u_{n+k}}\left(1-\frac{c_{n+k+1}}{b_{n+k+1}}\right) \geq 1-\left(1-\frac{c_{n+k+1}}{b_{n+k+1}}\right)=\frac{c_{n+k+1}}{b_{n+k+1}}.
\end{align*}
Therefore,  $\{u_{n_i-1}x\}\geq \frac{c_{m_i+1}}{b_{m_i+1}}\to1$, and so $\lim_{i\to \infty}\Vert u_{n_i-1}x\Vert=0$. 
\end{proof}

In order to refine the union in \eqref{Lk}, we replace the set $E$ by a decreasing sequence $\alpha= (A_k)$ of subsets of $\N$, by letting  
\begin{equation}\label{Ck}
C_k(\alpha):= \bigcup_{j = 0}^k (A_j + j).
\end{equation}
The sets $C_k(\alpha)$ form an increasing chain.  If $\alpha$ is the constant sequence given by $A_k=E$ for every $k\in \N$, then obviously
$C_k(\alpha)=S_k(E)$. 

In the next proof, fixed $\uu\in\A$, $\I$ a translation invariant free ideal of $\N$ and $x\in[0,1)$ with $S=\supp(x)$ and $S_b=\supp_b(x)$, we shall need the following property of $A \in\P(\N)\setminus\I$: 
\begin{itemize}
\item[$(\mathrm{S}_A)$] there exists a strictly decreasing sequence $\alpha=(A_k)$ of subsets of $A$ such that $A_k\subseteq_\I A$, $C_k(\alpha)\in\B_\uu$ and $C_k(\alpha)\subseteq S_b$ for every $k\in\N$.
\end{itemize}
Obviously, for every $\alpha$ as in $(\mathrm{S}_A)$, $S_k(A_k) \subseteq C_k(\alpha)\subseteq S_b$ for every $k\in\N$.
In particular, $A_k + k \subseteq C_k(\alpha)$, so the $b$-boundedness of $C_k(\alpha)$ implies that 
$A_k + k$ is $b$-bounded for every $k \in \N$.

\begin{lemma}\label{S3}
Let $\uu\in\A$, let $\I$ be a translation invariant free ideal of $\N$ and let $x\in[0,1)$ with $S=\supp(x)$ and $S_b=\supp_b(x)$. Then $\axuno \Rightarrow (\mathrm{U}_A)$ for every $A\in\B_\uu\setminus\I$ with $A \subseteq \supp(x)$.
\end{lemma}
\begin{proof}  Let $A\in\B_\uu\setminus\I$ with $A \subseteq S$. Arguing by contradiction, we assume that $(\mathrm{U}_A)$ fails, and prove that  $\neg (\mathrm{U}_A)\Rightarrow (\mathrm{S}_A)\Rightarrow (\mathrm{U}_A)$, a contradiction.

\smallskip
To prove that $\neg (\mathrm{U}_A)\Rightarrow (\mathrm{S}_A)$ we build inductively a strictly decreasing sequence $\alpha=(A_k)$ as in $(\mathrm{S}_A)$. Let $A_0:=A\cap S_b\subseteq_\I A$ by $\axuno$. Now fix $k\in\N$ and assume that $A_j$ with the properties in $(\mathrm{S}_A)$ are built for all $j\in\{0,\ldots,k\}$.
The fact that $A_k\subseteq_\I A$ and $A\not\in\I$ implies that $A_k\not\in\I$, so $A_k+k\not\in\I$ as well.
Since  $B:=A_k+k\subseteq S_b$ and $B\in\B_\uu$ by our inductive hypothesis, by $\axuno$ there exists $B'\subseteq_\I B$ with $\lim_{m\in B'}\frac{c_{m+1}+1}{b_{m+1}}=1$ and $B' \subseteq S_b$. Since $B- k \subseteq A_k \subseteq A$, our hypothesis $\neg (\mathrm{U}_A)$ implies $\neg (\mathrm{U}_{B- k})$. Hence, by Claim~\ref{DD} applied to $E= B-k=A_k$ and $F= B'$, we deduce that $B'+1\in\B_\uu$. Since also $B'$ is $b$-bounded and $B'\subseteq S_b$, we have by $\axuno$ that $B'+1\subseteq^\I S_b$. Replacing $B'$ by a suitable smaller set we may assume that $B'+1\subseteq S_b$. With $A_{k+1}:=B'-k$ we have $A_{k+1}+(k+1)=B'+1\subseteq S_b$. Moreover, from $A_{k+1}+k=B'\subseteq_\I B=A_k+k$ it follows that $A_{k+1}\subseteq_\I A_k$. This concludes the inductive construction of $\alpha=(A_k)$ as in $(\mathrm{S}_A)$. 

\smallskip
Now we prove that $(\mathrm{S}_A) \Rightarrow (\mathrm{U}_A)$. Let $\alpha=(A_k)$ be a sequence according to $(\mathrm{S}_A)$ and let $n_k\in A_k$ be such that $(n_k)$ is strictly increasing. We prove now that $\Vert u_{n_k-1}x\Vert\rightarrow 0$, so $(\mathrm{U}_A)$ holds. Fix an arbitrary $k\in \N$ and for the sake of simplicity write $n:=n_k$. Then for every $j\in\{0,\ldots,k\}$, one has $n+j\in A_k+j\subseteq A_j+j\subseteq C_k(\alpha)\subseteq S_b$, by \eqref{Ck}. Therefore,
$$u_{n-1}x\equiv_\Z u_{n-1}\sum_{i=n}^{n+k}\frac{b_i-1}{u_i}+u_{n-1}\sum_{i>n+k}\frac{c_i}{u_i} \geq u_{n-1}\left(\frac{1}{u_{n-1}}-\frac{1}{u_{n+k}}\right)=1-\frac{u_{n-1}}{u_{n+k}}\geq 1-2^{-k-1}.$$ This proves $\Vert u_{n_k-1}x\Vert\rightarrow 0$.
\end{proof}

\subsubsection{Chasing for $(\mathrm{U}_A)$ when $A \in \B_\uu\setminus\I$ and $A \cap \supp(x)=\emptyset$}

\begin{claim}\label{DDT}\label{DDU}
Let $\uu\in\A$, let $\I$ be a translation invariant free ideal of $\N$ and $x\in[0,1)$ with $S=\supp(x)$, let $k\in \N$ and $L$ be an infinite set such that $S_k(L)\cap S=\emptyset$. 
\begin{itemize}
   \item[(a)] If $(\mathrm{U}_L)$ does not hold, then $A' \cap (S-1) \in\I$ for every $A'\subseteq_\I L+k$ such that $\lim_{m\in A'}\frac{c_{m+1}}{b_{m+1}}=0$.
   \item[(b)] If $L+k\in\D_\uu$, then $(\mathrm{U}_L)$ holds.
\end{itemize}
\end{claim}
\begin{proof} 
(a) Let $A'\subseteq_\I L+k$ such that $\lim_{m\in A'}\frac{c_{m+1}}{b_{m+1}}=0$, assume that $A'' := A' \cap (S-1) \not\in\I$ and let $L'':= A'' -k\subseteq L$. Then $L''\not\in\I$, as well.
Moreover, $L'' + k + 1 = A'' + 1\subseteq S$, while for every $j\in\{0,\ldots,k\}$,
$L'' + j \subseteq S_k(L)$ and $S_k(L)\cap S=\emptyset$. Hence, for every $n\in L''$, $c_{n+k+1}\neq 0$ and $c_{n+j}=0$ for every $j\in\{0,\ldots,k\}$. Therefore,
\begin{align*}
u_{n-1}x&\equiv_\Z u_{n-1}\left(\frac{c_{n+k+1}}{u_{n+k+1}}+\sum_{i>n+k+1}\frac{c_i}{u_i}\right)\leq u_{n-1}\left(\frac{c_{n+k+1}}{u_{n+k+1}}+\frac{1}{u_{n+k+1}}\right)\leq u_{n+k}\left(\frac{c_{n+k+1}}{u_{n+k+1}}+\frac{1}{u_{n+k+1}}\right)\leq 2\frac{c_{n+k+1}}{b_{n+k+1}}.
\end{align*}
Since $n+k\in A'$ for $n\in L''$, this implies that $\lim_{n\in L''} \Vert u_{n-1}x\Vert=0$. As $L''\subseteq L$, we deduce that $(\mathrm{U}_L)$ holds.

(b) For $n\in L$, since $c_n=\ldots=c_{n+k}=0$ in view of $S_k(L)\cap S=\emptyset$, 
$$u_{n-1}x\equiv_\Z u_{n-1}\sum_{i>n+k}\frac{c_i}{u_i}\leq \frac{u_{n-1}}{u_{n+k}}\leq \frac{1}{b_{n+k}} .$$
Since $\lim_{n\in L}\frac{1}{b_{n+k}}=0$ by hypothesis, we obtain $\lim_{n\in L}\Vert u_{n-1}x\Vert=0$. Namely, $(\mathrm{U}_L)$ holds. 
\end{proof}

\begin{lemma}\label{S4} 
Let $\uu\in\A$, let $\I$ be a translation invariant free $P$-ideal of $\N$ and let $x\in[0,1)$ with $S=\supp(x)$. Then $\axdue \Rightarrow (\mathrm{U}_A)$  for every $A\in\B_\uu\setminus\I$ with $A \cap S=\emptyset$. 
\end{lemma}
\begin{proof} 
For a subset $A$ of $\N$, consider the property: 
\begin{itemize}
   \item[$(\mathrm{T}_A)$] there exists a decreasing sequence $(A_k)$ in $\P(\N)\setminus\I$ such that $A_k\subseteq A$, $A_k+k \in \B_\uu $ and $(A_k+k)\cap S=\emptyset$ for every $k\in\N$.
\end{itemize}
 Arguing by contradiction, we assume that $(\mathrm{U}_A)$ fails, and prove that  $\neg (\mathrm{U}_A)\Rightarrow (\mathrm{T}_A)\Rightarrow (\mathrm{U}_A)$, a contradiction.

First we verify that $\neg (\mathrm{U}_A)\Rightarrow(\mathrm{T}_A)$. Let $A_0:=A$. Fix $k\in\N$ and assume that all $A_j$ are built with the properties in $(\mathrm{T}_A)$, for every $j\in\{0,\ldots,k\}$. For $n\in A_k$ and $j\in\{0,\ldots,k\}$, we have $n+j\in A_j+j$ and $(A_j+j)\cap S=\emptyset$, thus $c_{n+j}=0$. By $\axdue$ there exists $A'\subseteq_\I A_k+k$ with $\lim_{m\in A'}\frac{c_{m+1}}{b_{m+1}}=0$. As $A_k+k\not \in \I$, this yields $A' \not \in \I$. Then $A'' := A' \cap (S-1) \in\I$, by Claim~\ref{DDT}(a). As $A'\not \in \I$, this entails $A' \setminus A''\not \in \I$ and $A'' + 1 = (A'+1) \cap S\in \I$. Since obviously $(A' \setminus A'')+1= (A'+1) \setminus (A''+1) \subseteq S^*$, we obtain
\begin{equation}\label{July21}
A_{k+1}' := (A'\setminus A'') -k \subseteq A_k, \ A_{k+1}'+k = A'\setminus A''\ 
\mbox{ and }A_{k+1}'+k + 1 = (A' \setminus A'')+1 \subseteq S^*.
\end{equation}
Next we find  $A_{k+1}\subseteq A_{k+1}'$ such that $A_{k+1}\not\in\I$ and $A_{k+1}+(k+1)\in \B_\uu$. In fact, if  this is not true, then by Lemma~\ref{Lemma2.8*} there exists $B\in\P(\N)\setminus\I$, $B\subseteq A_{k+1}'$ with $B+(k+1)\in \D_\uu$. Since for every $j\in\{0,\ldots,k\}$, $B\subseteq A_{k+1}'\subseteq A_k\subseteq A_j$, $B+j\subseteq S^*$, in view of \eqref{July21} and our inductive hypothesis; moreover, $B+k+1\subseteq S^*$ by \eqref{July21}, and hence $S_{k+1}(B)\cap S=\emptyset$.  
As $B$ is an infinite set with $S_{k+1}(B)\cap S=\emptyset$ and $B+k+1\in \D_\uu$, Claim~\ref{DDU}(b) (applied with $B=L$ and $k+1$ in place of $k$) implies that $(\mathrm{U}_B)$ holds, and hence $(\mathrm{U}_A)$ holds as well, a contradiction. Therefore, there exists a set $A_{k+1}$ as desired, i.e., the inductive definition of $(A_k)$ is finished, and $(\mathrm{T}_A)$ holds.

Now we prove that $(\mathrm{T}_A) \Rightarrow (\mathrm{U}_A)$.  Let $(A_k)$  be a sequence according to $\mathrm{(T_A)}$. Choose a strictly increasing sequence $(n_k)$ with $n_k\in A_k$; then $n_k\in S_k(A_k)$ and $S_k(A_k)\cap S=\emptyset$. Hence, for $n:=n_k$ we have
$$
u_{n-1}x\equiv_\Z u_{n-1}\sum_{i>n+k}\frac{c_i}{u_i}\leq u_{n-1}\frac{1}{u_{n+k}}\leq2^{-k-1}.
$$
It follows $\lim_{k\to\infty}\Vert u_{n_k-1}x\Vert=0$, so $(\mathrm{U}_A)$ holds.
\end{proof}

\subsubsection{Final stage of the sufficiency in Theorem~\ref{conjecture}}\label{final stage}

\begin{proposition}[Sufficiency in Theorem~\ref{conjecture}]\label{sufficiency} Let $\uu\in\A$, let $\I$ be a translation invariant free $P$-ideal of $\N$ and let $x\in[0,1)$ with $S=\supp (x)$. If $\ax\&\bx$ holds, then $\bar x\in t_\uu^\I(\T)$.
\end{proposition}
\begin{proof}  
In view of Lemma~\ref{UAholds}, to prove that $\bar x\in t_\uu^\I(\T)$ it is enough to check that  $(\mathrm{U}_A)$ holds, for any arbitrarily chosen $A\in\P(\N)\setminus\I$. If $A\in\D_\uu$, apply Lemma~\ref{L2.3}. Assume now that $A\in\B_\uu$. Since $A\not\in\I$, we may assume that either $A\subseteq S$ or $A\cap S=\emptyset$, in view of Remark \ref{case3}. If $A\subseteq S$, apply Lemma~\ref{S3}, and if $A\cap S=\emptyset$, apply Lemma~\ref{S4}. 

In the general case, in view of Lemma~\ref{Lemma2.8*}, either there exists $A'\subseteq A$ with $A'\in\B_\uu$ and $A'\not\in\I$, or there exists $A'\in\D_\uu$ with $A'\subseteq_\I A$. In both cases we conclude that always $(\mathrm{U}_A)$ holds, as $(\mathrm{U}_{A'})$ holds in view of the previous two cases considered above. 
\end{proof}

What is behind this proof is the fact that it suffices to check $(\mathrm{U}_A)$ not for all $A\in\P(\N)\setminus\I$, but replacing $A\in\P(\N)\setminus\I$
by some smaller $A'\in\P(\N)\setminus\I$ with more convenient properties in two directions: with respect to $S$
($A'\subseteq S$ or $A'\cap S=\emptyset$), or with respect to $\uu$ ($A'\in\B_\uu\cup\D_\uu$). 
The first direction is covered by Remark \ref{case3}, the second by Lemma~\ref{Lemma2.8*}. 

\begin{proof}[\bf Proof of Theorem~\ref{conjecture}]
If $\bar x\in t_\uu^\I(\T)$, then $\ax\&\bx$ holds in view of  Proposition~\ref{necessity}. 
Vice versa, if $\ax\&\bx$ holds, then $\bar x\in t_\uu^\I(\T)$ by Proposition~\ref{sufficiency}.
\end{proof}

\begin{proof}[{\bf Proof of Corollary~\ref{IcB}}] If $\bar x\in t_\uu^\I(\T)$, then $\bx$ holds by Proposition~\ref{necessity}. Vice versa, assume that $\bx$ holds and let $A\in\P(\N)\setminus\I$. By our hypothesis $\B_\uu \subseteq \I$,  Lemma \ref{Lemma2.8*} gives that $A\in \D_\uu + \I$, namely, $A$ is $b$-divergent mod $\I$, that is, there exists  $A'\subseteq_\I A$ such that $A' \in \D_\uu $. As $A\not\in\I$, we have also $A'\not\in\I$.
Then Lemma~\ref{L2.3} entails that $(\mathrm U_{A'})$ holds, so $(\mathrm U_{A})$ holds as well.
\end{proof}

\section{The $\I$-splitting property}\label{Splitsec}

Beyond the ideals related to a sequence $\uu\in \A$ from Definition~\ref{IdealiDuBu}, one can define also the ideal $\SS_\uu:= B_\uu + D_\uu$, namely, the ideal of all unions $S = B \cup D$, with $B \in B_\uu$ and $D\in \D_\uu$. 
Clearly, $\B_\uu \cap \D_\uu = \F in$. Therefore, $B \cap D$ is finite, and so after replacing $D$ by $D \setminus B$ one can have actually a partition  $S = B \sqcup D$. Therefore, the infinite members of $\SS_\uu$ are the $\uu$-splitting sets (see Definition~\ref{splittingdef}).

The next two diagrams show the inclusions relating the ideals $\B_\uu$, $\D_\uu$ and $\mathcal S_\uu$.

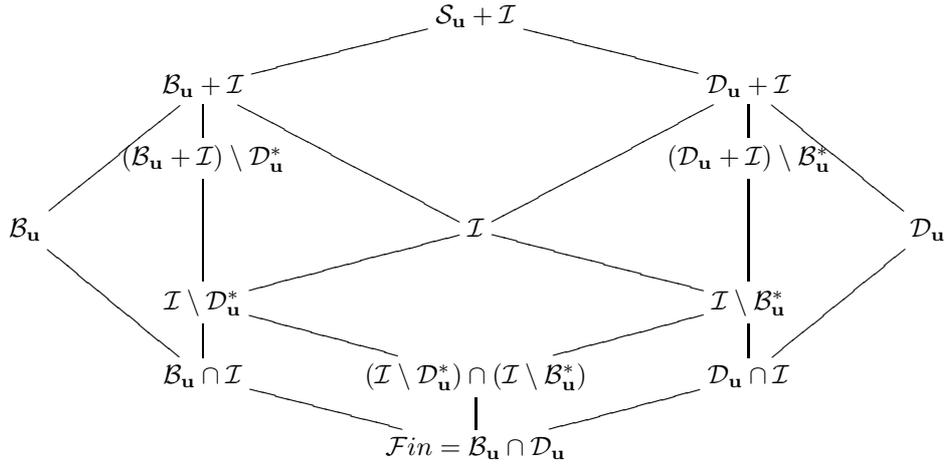
$$\xymatrix@R-1pc{
& & \mathcal S_\uu+\I \ar@{-}[dl]\ar@{-}[dr] & &\\
& \B_\uu+\I \ar@{-}[d]\ar@{-}[ddr]\ar@{-}[ddl]& & \D_\uu+\I \ar@{-}[d] \ar@{-}[ddl]\ar@{-}[ddr]& \\
& (\B_\uu+\I)\setminus\D_\uu^*\ar@{-}[dd] & &(\D_\uu+\I)\setminus\B_\uu^*\ar@{-}[dd] &\\
\B_\uu \ar@{-}[ddr]& & \I \ar@{-}[dr]\ar@{-}[dl]& & \D_\uu \ar@{-}[ddl] \\
& \I\setminus \D_\uu^*\ar@{-}[dr] \ar@{-}[d]& & \I\setminus \B_\uu^*\ar@{-}[dl] \ar@{-}[d] & \\
& \B_\uu\cap \I\ar@{-}[dr] & (\I\setminus \D_\uu^*)\cap (\I\setminus \B_\uu^*)\ar@{-}[d] & \D_\uu\cap\I \ar@{-}[dl]&\\
& & \F in=\B_\uu\cap \D_\uu & &
}$$

When $\I=\P(\N)$, the above diagram  simplifies as follows.

$$
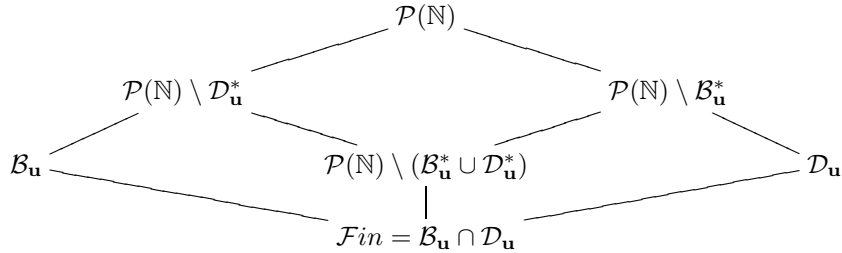
\xymatrix@R-1pc{
 & & \P(\N) \ar@{-}[dr]\ar@{-}[dl]& & \\
& \P(\N)\setminus \D_\uu^*\ar@{-}[dr]\ar@{-}[dl] & & \P(\N)\setminus \B_\uu^* \ar@{-}[dr]\ar@{-}[dl]& \\
\B_\uu\ar@{-}[drr] & & \P(\N)\setminus (\B_\uu^*\cup \D_\uu^*)\ar@{-}[d] & & \D_\uu\ar@{-}[dll]\\
& & \F in=\B_\uu\cap \D_\uu & &
}$$

A set $A\in\mathcal S_\uu+\I$ has the form $A=B\cup D\cup X$, with $B\in\B_\uu$, $D\in \D_\uu$ and $X\in \I$. Equivalently, $A$ has the form $A=B\cup D$, with $B$ $b$-bounded mod $\I$ and $D$ $b$-divergent mod $\I$.  Then the sets $A\in\mathcal S_\uu+\I$ are exactly those with the $\uu_\I$-splitting property.

\smallskip
As in \cite{BCsurvey}, fixed a free ideal $\I$ of $\N$, call an ideal $\mathcal J$ of $\N$ \emph{principal mod $\I$} if there exists $B\in\P(\N)$ such that $\mathcal J=\P(B)+\I$.

\begin{proposition}\label{corrected} 
For $\uu\in\A$ and $\I$ a free $P$-ideal of $\N$, consider the following conditions:
\begin{itemize}
        \item[(a)]  $\I^{bd}=\B_\uu+\I$ is principal mod $\I$; 
        \item[(b)]  $\I^{div}=\D_\uu+\I$ is principal mod $\I$; 
        \item[(c)] $\mathbf u$ has the $\I$-splitting property.
\end{itemize}
Then (a)$\Leftrightarrow$(c)$\Rightarrow$(b).
\end{proposition}
\begin{proof}
(c)$\Rightarrow$(a)\&(b) Assume that $\N_+=B\sqcup D\sqcup X$ where $B\in\B_\uu$, $D\in\D_\uu$ and $X\in\I$. Then $\B_\uu\subseteq\P(B)+\I$
and $\D_\uu\subseteq \P(D)+\I$, that is, $\I^{bd}=\B_\uu+\I=\P(B)+\I$ and $\I^{div}=\D_\uu+\I=\P(D)+\I$.

(a)$\Rightarrow$(c) Suppose that $\I^{bd}=\B_\uu+\I=\P(B)+\I$ for some $B\in\P(\N)$; we may assume without loss of generality that $B\in\B_\uu$. If $\P(B)+\I=\P(\N)$, take $D=\emptyset$ and $B=\N_+$.
Otherwise, letting $D:= \N_+\setminus B$, it is enough to verify that $D$ is $b$-divergent mod $\I$. To this end, by Lemma \ref{Lemma2.8*}, we have to see that $\B_\uu \cap \P(D) \subseteq \I$. Indeed, if $D' \subseteq D$ and $D' \in \B_\uu \subseteq \P(B)+\I$, then $D' = B' \cup I$, with $B' \subseteq B$ and $I \in \I$. 
Since $D' \cap B' =D \cap B =\emptyset$, we deduce that $D' = I \in \I$. Therefore, $\B_\uu \cap \P(D) \subseteq \I$. 
\end{proof}

The next example is a counterexample to the implication (b)$\Rightarrow$(c) in Proposition~\ref{corrected}.

\begin{example}\label{bnotc}
Let $\I=\I_d$ and let $\uu\in\A$ be determined by, for every $n\in\N_+$,
$$b_n:=\begin{cases}
        2, & \text{if $n$ is odd,}\\[2pt]
        v_2(n)+1, & \text{if $n$ is even,}
       \end{cases}$$
where $v_2(n)$ denotes the $2$-adic valuation of $n$. 

We prove below that $\I^{div}$ is principal modulo $\I$, while $\uu$ does not have the $\I$-splitting property. In particular, the implication (b)$\Rightarrow$(c) in Proposition~\ref{corrected} fails.

For every $m\geq 2$ and $n\in\N_+$ one has $b_n>m$ if and only if $n$ is even and $v_2(n)\geq m$, that is,
\begin{equation}\label{eq:levels}
  \{n\in\N_+ : b_n>m\}=2^{m}\N_+ ,
  \quad\text{so}\quad
  d\bigl(\{n\in\N_+ : b_n>m\}\bigr)=2^{-m}>0 ,
\end{equation}
where we used $|2^{j}\N_+\cap[1,n]|=\lfloor n/2^{j}\rfloor$, hence $d(2^{j}\N_+)=2^{-j}$ for every $j\in\N$.

To see that $\D_\uu\subseteq\I$, take any $X\in \D_\uu$ and fix $m\ge 2$. Since $\lim_{n\in X}b_n=\infty$, only finitely many $n\in X$ satisfy $b_n\le m$, so by \eqref{eq:levels} the set $X\setminus 2^{m}\N_+$ is finite.
Then, there exists $N_m\in\N_+$ such that $X\setminus 2^{m}\N_+\subseteq[1,N_m]$. Consequenly, for every $n>N_m$,
$$|X\cap[1,n]|\leq N_m+|2^{m}\N_+\cap[1,n]|\le N_m+\frac{n}{2^{m}},$$
therefore $d(X)\leq 2^{-m+1}$. As $m$ is arbitrary, $d(X)=0$, i.e., $X\in\I$. Hence, $\I^{div}=\D_\uu+\I=\I=\P(\emptyset)+\I$ is principal modulo $\I$.

Now we verify that (c) fails, that is, $\uu$ has not the $\I$-splitting property. Assume for a contradiction that $\N_+=B\sqcup D$ witnesses the $\I$-splitting property for $\uu$.
If $D\neq\emptyset$, then $D\notin\I$ and $D$ is $b$-divergent mod $\I$, so there exists a $b$-divergent $D'\subseteq D$ with $D\setminus D'\in\I$. By the previous step, $D'\in\I$, hence also $D\in\I$, a contradiction.
Therefore, $D=\emptyset$ and $B=\N_+\notin\I$, so there exists a $b$-bounded $B'\subseteq\N_+$ with $\N_+\setminus B'\in\I_d$. Let $m\ge 2$ be such that $b_n\le m$ for every $n\in B'$. By~\eqref{eq:levels}, $2^{m}\N_+\subseteq \N_+\setminus B'$, and $d(2^{m}\N_+)=2^{-m}>0$ gives $\N_+\setminus B'\notin\I$, a contradiction.
\end{example}

%In the following immediate corollary we find the known splitting property from \cite{DiD} in terms of $\B_\uu$ and $\D_\uu$.
 
%\begin{corollary}\label{Lemma:superato}
%For $\uu\in\A$, the following conditions are equivalent: 
%\begin{itemize}
%\item[(a)] $\B_\uu$ is principal modulo $\F in$; 
%\item[(b)] $\D_\uu$ is principal modulo $\F in$; 
%\item[(c)] $\mathbf u$ has the splitting property.
%\end{itemize}
%\end{corollary}

\begin{remark}\label{remark13}
Assume that $\I$ is a free ideal of $\N$ and that the sequence $\uu\in\A$ has the $\I$-splitting property witnessed by the partition $\N_+=B\sqcup D$.
If $D=\emptyset$, namely $B=\N_+$, this means that there exists $B'\subseteq_\I \N_+$ such that $B'$ is $b$-bounded, i.e., $\uu$ is $b$-bounded mod $\I$. 
Analogously when $B=\emptyset$, this means that $\uu$ is $b$-divergent mod $\I$. 
\end{remark}

Here,  under the assumption that $\uu\in\A$ has the $\I$-splitting property with respect to a translation invariant free $P$-ideal $\I$ of $\N$, first we prove the equivalence $\axuno\&\axdue\&\bx\Leftrightarrow  \unox\&\duex\&\trex$ in Proposition \ref{unoduetre}, where for one implication we follow the proof of \cite[Theorem~3.8]{DG1} (which treats the case of the ideal $\I_\alpha$). Then we give a proof of Corollary~\ref{ThGh:May29} as a direct consequence also of Theorem~\ref{conjecture}.

\begin{proposition}\label{unoduetre} 
Let $\uu\in\mathcal A$, let $\I$ be a translation invariant free $P$-ideal of $\N$ and $x \in[0,1)$ with $S=\supp(x)$ and $S_b=\supp_b(x)$. If $\mathbf u$ has the $\I$-splitting property witnessed by the partition $\N_+=B\sqcup D$, then $\axuno\Leftrightarrow\unox$, $\axdue\Leftrightarrow\duex$ and $\bx\Leftrightarrow\trex$.
\end{proposition}
\begin{proof} 
At least one between $B$ and $D$ is not in $\I$ and either $B=\emptyset$ or $B\not\in\I$ and either $D=\emptyset$ or $D\not\in\I$. In case $B$ (resp., $D$) is not empty, let $B' \subseteq_\I B$ be $b$-bounded (resp., $D'\subseteq_\I D$ be $b$-divergent).

\smallskip
$\axuno\Rightarrow\unox$  
By Remark \ref{Rem:June14}, $\unox$ is vacuously satisfied when $B\cap S \in \I$. This is why we are left with the case $B\cap S\not\in\I$. Then $B'\cap S \subseteq_\I B\cap S$ is $b$-bounded as well and $B'\cap S\not \in \I$.
Now $\axuno$, applied to $B'\cap S$, entails that $(B'\cap S)+1 \subseteq^{\I} S$, $B'\cap S \subseteq^{\I} S_b$ and the existence of a 
$C \subseteq_{\I} B'\cap S$ such that $\lim_{n \in C}\frac{c_{n+1}+1}{b_{n+1}}=1$. 
Since $B'\cap S \subseteq_\I B\cap S$, so in particular $B\cap S\subseteq^\I B'\cap S$, from Lemma~\ref{New:Lemma}(b)-(d) it follows that $(B\cap S)+1 \subseteq^{\I} S$, $B\cap S \subseteq^{\I} S_b$ and $C \subseteq_{\I} B\cap S$. Therefore, $\unox$ holds.

\smallskip
$\unox\Rightarrow\axuno$  
Let $A\in\B_\uu\setminus \I$, so $A\subseteq^\I B$. Assume that $A\subseteq^\I S$, that is $A \setminus S\in\I$. Then $A\cap S\not\in\I$,  
and so $B\cap S\not\in\I$ since $A\cap S\subseteq^\I B\cap S$. By $\unox$, $(B\cap S)+1 \subseteq^{\I} S$, $B\cap S \subseteq^{\I} S_b$ and there exists $C \subseteq_\I B\cap S$ such that $\lim_{n \in C}\frac{c_{n+1}+1}{b_{n+1}}=1$.  Since $A\setminus (B\cap S)=(A\setminus B)\cup (A\setminus S)\in\I$, we have that $A\subseteq^\I B\cap S,$ thus $A+1\subseteq^\I (B\cap S)+1.$ Now, applying twice  Lemma~\ref{New:Lemma}(d), since $A+1\subseteq^\I (B\cap S)+1 \subseteq^{\I} S$ 
and $A\subseteq^\I B\cap S\subseteq^{\I} S_b$, we get $A+1\subseteq^\I S$ and  $A\subseteq^\I S_b$. Moreover, letting $A'=C\cap A$, we have also $A'\subseteq_\I A$ (as $A\setminus A' = A\setminus C \subseteq (A \setminus (B\cap S)) \cup ((B\cap S) \setminus C)\in\I$) and $\lim_{n\in A'}\frac{c_{n+1}+1}{b_{n+1}}=1$.

\smallskip
$\axdue\Rightarrow\duex$
Assume that $B\setminus S\not\in\I$, so that $B'\setminus S\not\in\I$ as well. Since $(B'\setminus S)\cap S=\emptyset$ and $B'\setminus S$ is $b$-bounded, by $\axdue$ applied to $B'\setminus S$ there exists $C\subseteq_\I B'\setminus S$ (so, $C=C\setminus S \subseteq_\I B\setminus S$ as well) such that $\lim_{n\in C}\frac{c_{n+1}}{b_{n+1}}=0$, namely, $\duex$ holds.

\smallskip
$\duex\Rightarrow\axdue$ 
Let $A\in\B_\uu\setminus \I$, so $A\subseteq^\I B$, and assume that $A\cap S\in\I$. Then $A\setminus S\not\in\I$ (as otherwise $A=(A\setminus S) \cup (A\cap S) \in \I$), $A\setminus S\subseteq^\I B\setminus S$ (since $(A\setminus S) \setminus (B\setminus S) \subseteq A\setminus B \in \I)$, hence  $B\setminus S\not\in\I$. By $\duex$, there exists $C\subseteq_{\I} B\setminus S$ such that $\lim_{n \in C}\frac{c_{n+1}}{b_{n+1}}=0$. Let $A'=A\cap C$. Then $A'\subseteq_\I A$ as $A\cap S\in\I$ and $A\cap C\subseteq_\I A\cap (B\setminus S)\subseteq_\I A\setminus S$, since $A\setminus S\subseteq^\I B\setminus S$. Moreover, $A'\cap S=\emptyset$ since $A'\subseteq B\setminus S$, and $\lim_{n\in A'}\frac{c_{n+1}}{b_{n+1}}=0$ since $A'\subseteq C$.

\smallskip 
$\bx\Rightarrow\trex$ If $D\cap S\not\in\I$, then $D'\cap S\not\in\I$ and $D'\cap S$ is $b$-divergent. By $\bx$ applied to $D'\cap S$, there exists $E\subseteq_\I D'\cap S$, such that $\lim_{n\in E}\varphi\left(\frac{c_n}{b_n}\right)=0$, that is, $\trex$ holds since $D'\cap S\subseteq_\I D\cap S$ and so $E\subseteq_\I D\cap S$ as well.

$\trex\Rightarrow\bx$ 
If $A\in\D_\uu\setminus\I$, then $A\subseteq^\I D$.  
So, if $A\cap S\in\I$, then $A\setminus S \subseteq_\I A$   and 
$\bx$ trivially holds, since $\lim_{n \in A\setminus S} \varphi\left(\frac{c_n}{b_n}\right)=0$, as $c_n=0$ for every $n\in A\setminus S$.

If $A\cap S\not\in\I$, from $A\cap S\subseteq^\I D\cap S$ we conclude that $D\cap S\not\in\I$.  So, in view of $\trex$ there exists $E\subseteq_\I D\cap S$ with $\lim_{n \in E} \varphi\left(\frac{c_n}{b_n}\right)=0$. 
 Then $A \cap E \subseteq_\I A \cap D\cap S$ (as $E\subseteq_\I D\cap S$) and $A \cap D\cap S \subseteq_\I A\cap S$ (as $A\cap S \subseteq^{\I} D\cap S$), hence $A \cap E \subseteq_\I A\cap S$ by Lemma~\ref{New:Lemma}(d). 
Therefore  for $A':= (A\cap E)\sqcup (A\setminus S)$ we have $A' \subseteq_\I A = (A\cap S) \sqcup (A\setminus S)$. Moreover, $\lim_{n \in A'} \varphi\left(\frac{c_n}{b_n}\right)=0$ since $A'\subseteq (E\sqcup A)\setminus S$.
\end{proof}

\begin{proof}[\bf Proof of Corollary~\ref{ThGh:May29}]
For  $\uu\in\mathcal A$, a translation invariant free $P$-ideal $\I$ of $\N$ and $x \in[0,1)$
we have to prove that if $\mathbf u$ has the $\I$-splitting property, then $\bar x\in t_\uu^\I(\T) \Leftrightarrow \unox\&\duex\&\trex$ holds.
This equivalence follows from Theorem~\ref{conjecture} and Proposition~\ref{unoduetre}.
\end{proof}

\section{Sufficient conditions for $\bar x \notin t^{\I}_\uu(\T)$}\label{notinTsec}

Following \cite[\S 5]{DDB}, in this section we address the natural problem of finding specific elements not in $t_\uu^\I(\T)$.

\subsection{When $\uu$ is $b$-bounded}\label{notb}

First we deal with the case when $\uu\in\mathcal A$ is $b$-bounded. 
In the next result we give a correct and generalized version of \cite[Proposition~5.1]{DDB}. In view of Lemma~\ref{ginfty}, in (a) $c(\supp(x))<\infty$ implies that $\supp(x)^*\not\in\I$ and in (b) $g(\supp(x))<\infty$ implies that $\supp(x)\not\in\I$.

\begin{theorem}\label{Scg}
Let $\uu\in\A$ be $b$-bounded, let $\I$ be a translation invariant free ideal of $\N$ and let $x\in[0,1)$. Then $\bar x\not\in t_\uu^\I(\T)$ whenever at least one of the following conditions hold:
\begin{itemize}
   \item[(a)] $\supp(x)\not\in\I$ and $c(\supp(x))<\infty$;
   \item[(b)] $\supp(x)^*\not\in\I$ and $g(\supp(x))<\infty$.
\end{itemize}
\end{theorem}
\begin{proof}
Since $\uu$ is $b$-bounded, let $M\in\N_+$ such that $b_n\leq M$ for every $n\in\N_+$.
From the hypothesis $\supp(x)\not\in\I$ (respectively, $\supp(x)^*\not\in\I$) we deduce that $\I$ is a proper ideal, and that $\supp(x)$ is infinite and not cofinite. Hence, we can write $\supp(x)=\bigcup_{n\in\N}[k_n,l_n]$ with $k_n\leq l_n<k_{n+1}-1$ for every $n\in\N$.
Furthermore, for the sake of convenience and since $\I$ is free, we assume that $0 \not \in \supp(x)$. 

\smallskip
(a) Assume that $\supp(x)\not\in\I$ and $c(\supp(x))<\infty$.  Since $c(\supp(x))$ is finite, there exists $\ell\in\N_+$ such that $c(\supp(x))=\ell$, which means that $l_n-k_n+1\leq\ell$ for every $n\in\N$. 

Let $E = \left\{n\in \N: \{u_nx\}\in\left[\frac{1}{M},1-\frac{1}{M^{\ell+1}}\right]\right\}$. To prove that $\bar x\not\in t_\uu^\I(\T)$ it is enough to check that $E \not \in \I$. To this end we show that $\supp(x) - 1 \subseteq E$, since $\supp(x)-1\not\in\I$, being $\I$ translation invariant. 
Indeed, if $n\in\supp(x) - 1$, then there exists $m\in\N$ such that $n+1\in[k_m,l_m]$.
Hence, $$\{u_nx\}\geq\frac{u_n}{u_{n+1}}=\frac{1}{b_{n+1}}\geq\frac{1}{M}.$$
On the other hand, $$\{u_nx\}\leq u_n\left(\sum_{\supp(x) \ni i>n}
\frac{b_i-1}{u_i}-\frac{1}{u_{l_m+1}}\right)\leq 1-\frac{u_n}{u_{{l_m}+1}}\quad\text{and}\quad \frac{u_n}{u_{{l_m}+1}}=\frac{1}{b_{n+1}\ldots b_{l_m+1}}\geq\frac{1}{M^{\ell+1}},$$ since $l_m+1-(n+1)+1=l_m-n+1\leq l_m-k_m+2\leq\ell+1$.
Then $\{u_nx\}\leq 1-\frac{1}{M^{\ell+1}}$, so $n \in E$. 
This finishes the proof of the inclusion $\supp(x) - 1 \subseteq E$, which ensures that $\bar x\not\in t_\uu^\I(\T)$.

\smallskip
(b) Assume that $\supp(x)^*\not\in\I$ and $g(\supp(x))<\infty$. Write $\supp(x)=\bigcup_{n\in\N}[k_n,l_n]$ with $k_n\leq l_n<k_{n+1}-1$ for every $n\in\N_+$. Since $g(\supp(x))$ is finite, there exists $\ell\in\N_+$ such that $g(\supp(x))=\ell$, which means that $k_{n+1}-l_n-1\leq\ell$ for every $n\in\N$.

Let $F= \left\{n\in \N: \{u_nx\}\in\left[\frac{1}{M^{\ell+1}},\frac{1}{2}\right]\right\}$. To prove that $\bar x\not\in t_\uu^\I(\T)$ it is enough to check that $F\not\in \I$. To this end we show that $S'-1\subseteq F$, where $S'=\bigcup_{n\in\N}(l_n,k_{n+1})$;  
since $\supp(x)^*\not\in\I$, so also $S'=\bigcup_{n\in\N}(l_n,k_{n+1})\not\in\I$, and then $S'-1\not\in\I$ as $\I$ is translation invariant.
So, let $n\in\N$ with $n+1\in\supp(x)^*$. Then there exists $m\in\N$ such that $n+1\in(l_m,k_{m+1})$. Since 
$$\{u_nx\}=u_n \sum_{\stackrel{i\in\supp(x)}{i\geq k_{m+1}}}\frac{c_i}{u_i}\geq \frac{u_n}{u_{k_{m+1}}}=\frac{1}{b_{n+1}\ldots b_{k_{m+1}}},$$ 
we get that $\{u_nx\}\geq\frac{1}{M^{\ell+1}}$, since $k_{m+1}-(n+1) + 1 \leq k_{m+1}-l_m\leq\ell+1$.
On the other hand, since $n+1\not\in\supp(x)$, $$\{u_nx\}\leq u_n\sum_{\stackrel{i\in\supp(x)}{i\geq n+2}}\frac{b_i-1}{u_i}=\frac{u_n}{u_{n+1}}=\frac{1}{b_{n+1}}\leq\frac{1}{2}.$$
Hence, $n\in F$. This finishes the proof of the inclusion $S'-1 \subseteq F$, which ensures that $\bar x\not\in t_\uu^\I(\T)$.
\end{proof}

The above mentioned~\cite[Proposition~5.1]{DDB} deals with the case $\I=\I_d$, but the hypothesis $\supp(x)\not\in\I_d$ is missed in its statement, and the necessity of such an hypothesis is witnessed by Lemma~\ref{Lemma2.2}. Moreover, it states that under the assumption that there exists $\ell\in\N$ such that $c_n(\supp(x))-1\leq \ell$ and $g_n(\supp(x))+1\leq \ell$ for every $n\in\N$, then $\bar x\not\in t_\uu^s(\T)$. The following example shows that at least $g(\supp(x))\geq1$ has to be imposed.

%\begin{example}
%For every $n\in\N$,  let $l_n= 2n$ and $m_n = 2n+1$, so $\N=\{0=l_1<m_1<l_2<m_2< \ldots\}$. Then, for any $x\in [0,1)$ with $\supp(x) = \bigcup_{n=1}^\infty [l_n,m_n]$, the hypothesis that there exists $\ell\in\N$ such that $c_n(\supp(x))-1=m_n-l_n\leq \ell$ and $g_n(\supp(x))+1=l_{n+1}-m_n\leq \ell$ of~\cite[Proposition 5.1]{DDB} is satisfied with $\ell=1$.
%
%Since for such $x$, $\supp(x)=\N$, the conditions $\ix$ and  $\axdue$ are satisfied, so $\bar x \in t_\uu^\I(\T)$ whenever $\supp_b(x)\subseteq_\I \supp(x)$ by Theorem~\ref{Nuovo:Th}, a contradiction with the conclusion of~\cite[Proposition 5.1]{DDB}. 
%
%(For a specific example, take $\uu$ with $u_n = 3^n$ for every $n\in\N$, and $x$ defined by $c_n = 1$ when $n=k^2$ is a perfect square, and $c_n = 2$ otherwise.) 
%\end{example}

\begin{example}
%\NBA\footnote{Corretto e ampliato}
For every $n\in\N$, let $l_n=2n+1$ and $m_n=2n+2$, so that $\N_+=\{l_0<m_0<l_1<m_1<\dots\}$ and $\bigcup_{n\in\N}[l_n,m_n]=\N_+$. Let $\uu\in\mathcal A$
be $b$-bounded, let $\I$ be a translation invariant free ideal of $\N$ and let $x\in[0,1)$ with $\supp(x)=\bigcup_{n\in\N}[l_n,m_n]=\N_+$.

Since $m_n-l_n=1$ and $l_{n+1}-m_n=1$ for every $n\in\N$, the hypothesis of~\cite[Proposition 5.1]{DDB}, namely, the existence of $\ell\in\N$ such that $c_n(\supp(x))-1\le\ell$ and $g_n(\supp(x))+1\le\ell$ for every $n\in\N$, is satisfied with $\ell=1$, as long as the intervals $[l_n,m_n]$ are not required to be the \emph{maximal} intervals of $\supp(x)$ (here they are adjacent, so the gaps between them are empty; note that $\supp(x)=\N_+$ is cofinite, so $g(\supp(x))=1$).

On the other hand, $\supp(x)+1\subseteq\supp(x)$, so $\ix$ holds; moreover, $\axdue$ is vacuously satisfied, since $A\cap\supp(x)=A\setminus\{0\}$ for every $A\subseteq\N$, so $A\cap\supp(x)\in\I$ yields $A\in\I$, as $\I$ is free. 
Since $\uu$ is $b$-bounded, $\supp(x)$ is $b$-bounded, so Theorem~\ref{Nuovo:Th} applies and gives $\bar x\in t^{\I}_{\uu}(\T)$ whenever $\iix$ holds. This contradicts the conclusion of~\cite[Proposition 5.1]{DDB}.

(For a specific example, take \texttt{$\I=\mathcal I_d$,} $\uu$ with $u_n = 3^n$ for every $n\in\N$, and $x$ defined by $c_n = 1$ when $n=k^2$ is a perfect square, and $c_n = 2$ otherwise; then $\supp(x)\setminus\supp_b(x)$ is the set of all perfect squares, which has density $0$.)
\end{example}

The following are direct consequences of Theorem~\ref{Scg}.

\begin{corollary}
Let $\uu\in\A$ be $b$-bounded, let $\I$ be a translation invariant free ideal of $\N$ and let $x\in[0,1)$ with $\supp(x)\not\in\I$ and $\supp(x)^*\not\in\I$. If $\bar x\in t_\uu^\I(\T)$, then $c(\supp(x))=g(\supp(x))=\infty$.
\end{corollary}

\begin{corollary}
Let $\uu\in\A$ be $b$-bounded, let $\I$ be a translation invariant free ideal of $\N$ and let $x\in[0,1)$.
\begin{itemize}
\item[(a)]  If $c(\supp(x))<\infty$, then $\bar x\in t_\uu^\I(\T)$ if and only if $\supp(x)\in\I$.
\item[(b)] If $g(\supp(x))<\infty$ and $\bar x\in t_\uu^\I(\T)$, then $\supp(x)^*\in\I$. 
\end{itemize}
\end{corollary}
\begin{proof}
(a) If $\bar x\in t_\uu^\I(\T)$ and $c(\supp(x))<\infty$, then $\supp(x)\in\I$ by Theorem~\ref{Scg}. If $\supp(x)\in\I$, then $\bar x\in t_\uu^\I(\T)$ by Lemma~\ref{Lemma2.2}.

(b) If $\bar x\in t_\uu^\I(\T)$ and $g(\supp(x))<\infty$, then $\supp(x)^*\in\I$ by Theorem~\ref{Scg}. 
\end{proof}

An equivalence in item (b) of the above corollary is not available as the next example shows: 

\begin{example}\label{supp*no} To see that $\supp(x)^*\in\I$ does not imply $\bar x\in t_\uu^\I(\T)$ even under the assumption $g(\supp(x))<\infty$, pick a proper free ideal $\I$  of $\N$, $\uu=(3^n)$ and $x=\sum_{n=1}^\infty \frac{1}{3^n}$. 
Then $u_nx\equiv_\Z \frac{1}{2}$ for every $n\in\N$, hence $\bar x\not\in t_\uu^\I(\T)$. On the other hand, $\supp(x)=\N_+$, so $g(\supp(x))=1<\infty$ and $\supp(x)=\N_+\in\I^*$.
\end{example}

\subsection{General case}\label{notgc}

The next result extends both \cite[Proposition 2.13]{Ghosh} and  \cite[Proposition 5.2]{DDB}. With respect to \cite[Proposition 5.2]{DDB} we relax the condition ``$b$-divergent" on the sequence $\uu\in\A$, with a more careful, yet, simpler, proof. 

\begin{proposition}\label{Ghosh:Prop}
Let $\uu\in\mathcal A$ and let $\I$ be a translation invariant free ideal of $\N$. Let $x\in[0,1)$ with $\supp(x)\not\in\I$. If there exist $m_1, m_2 \in \R$ with $0<m_1\leq m_2 < \frac{1}{2}$ and for all $n \in\supp(x)$, $\frac{c_n}{b_n} \in [m_1,m_2]$, then $\bar x \notin t^{\I}_\uu(\T)$.
\end{proposition}
\begin{proof} 
Let $x \in [0,1)$. Without loss of generality we assume that $0\notin \supp(x)$, and let $A_1=\supp(x)-1$.
Then $A_1 \notin \I$, as $\I$ is translation invariant. Moreover, for all $n\in A_1$, as $c_k\geq m_1  b_k$ when $k \in \supp(x)$,
$$u_nx \equiv_\Z u_n \cdot\sum\limits_{\stackrel{k\in \supp(x)}{k>n}}\frac{c_k}{u_k}\geq u_n\cdot \sum\limits_{\stackrel{k\in \supp(x)}{k>n}}\frac{m_1 b_k}{u_k}= u_n\cdot \sum\limits_{\stackrel{k\in \supp(x)}{k>n}}\frac{m_1}{u_{k-1}} \geq u_n \cdot \frac{m_1}{u_n}=m_1;$$
therefore, $\{u_n x\}\geq m_1$.

As $u_k = u_n b_{n+1} \ldots b_k$ for $k > n$, and $b_n \geq 2$ for every $n \in \N_+$, we deduce that $\frac{u_n}{u_k} \leq \frac{1}{2^{k-n}}$ for $n\leq k$ in $\N$. Therefore, for $n \in A_1$,  as $c_k \leq m_2  b_k$ for every $k\in\N_+$, 
$$\{u_nx\}=u_n\cdot \sum\limits_{\stackrel{k\in \supp(x)}{k>n}}\frac{c_k}{u_k} \leq  u_n \cdot \sum\limits_{\stackrel{k\in \supp(x)}{k>n}}\frac{m_2 b_k}{u_k}=
m_2\cdot\sum\limits_{\stackrel{k\in \supp(x)}{k>n}}\frac{u_n}{u_{k-1}} \leq m_2 \cdot \sum_{k>n}\frac{1}{2^{k-1-n}}=2m_2<1;$$
hence, $\{u_nx\}\leq 2m_2$.

We have obtained that for every  $n \in A_1$, $\{u_nx\} \in [m_1, 2m_2]$. As $A_1 \notin \I$, this yields $\bar x \notin t^{\I}_\uu(\T)$. 
\end{proof}

This proposition was given in \cite[Proposition 2.13]{Ghosh}, asking $\I$ to be in addition analytic, but without any proof; only a reference to  \cite[Proposition 5.2]{DDB} is given saying ``The proof follows from similar line of arguments as in \cite[Proposition 5.2]{DDB} and so is omitted." A look at  \cite[Proposition 5.2]{DDB} shows that the sequence $\uu$ is assumed to be $b$-divergent (and this is used essentially in the proof), whereas no such restraint is imposed in \cite[Proposition 2.13]{Ghosh}. 

\smallskip
The following consequence of Proposition~\ref{Ghosh:Prop} is \cite[Corollary 2.14]{Ghosh} in which we relax the hypotheses on the ideal. It generalizes and reinforces \cite[Corollary~5.3]{DDB}.

\begin{corollary}\label{Corollary2.14}
Let $\uu\in\mathcal A$, let $\I$ be a translation invariant free $P$-ideal of $\N$ and $B \subseteq \N$. Then there exists $x \in [0,1)$ with $\supp(x) \subseteq B$ such that $\bar x \notin t^{\I}_\uu(\T)$ if and only if $B \not\in\I$.
\end{corollary}
\begin{proof}
If $x \in [0,1)$ has $\supp(x) \subseteq B$ and $\bar x \notin t^{\I}_\uu(\T)$, then $\supp(x) \notin \I$ by Lemma \ref{Lemma2.2}, hence also  $B \notin \I$.

Now assume $B\not\in\I$ and write $B=\{n_1< n_2< \dots < n_k< \dots  \}$. If $B$ is $b$-bounded, let  $S=\{n_1<n_3< \ldots < n_{2k-1}< \ldots  \}$ and $T=\{n_2< \ldots < n_{2k} < \ldots \}$. It follows that $S \sqcup T=B$.  As $B \notin \I$, at least one of these two sets does not belong to $\I$. Assume without loss of generality that $S \notin \I$. Then take any $x\in[0,1)$ with $\supp(x)=S$. Since $\I$ is translation invariant, $S+1 \notin \I$. As $(S+1) \cap S=\emptyset$, we have that $(S+1) \setminus S=S+1\not\in\I$, namely, $\ix$ does not hold. 
Hence, $\bar x \notin t^{\I}_\uu(\T)$ by Theorem~\ref{Nuovo:Th}.

If $B$ is not $b$-bounded and there exists $A \subseteq B$ such that $A\in\B_\uu\setminus\I$, we can argue as above replacing $B$ with $A$. If such an $A$ does not exist, by  Lemma \ref{Lemma2.8*}, there exists $A \subseteq_{\I} B$ (so $A\not\in\I$) with $A\in\D_\uu$.  Since $A$ is $b$-divergent, there exists $k \in \N$ such that $b_n>2$ for every $n\in A$ with $n>k$, hence $\left\lfloor \frac{b_n}{3} \right\rfloor \neq 0$. Let $x\in[0,1)$ with $\supp(x)=A \cap \{n \in \N: n>k \}$ and $c_n=\left\lfloor \frac{b_n}{3} \right\rfloor$. 
Then there exist $m_1,m_2\in\R$ with $0<m_1 \leq m_2<1/2,$ such that $m_1 \leq \frac{c_n}{b_n} \leq m_2$ for every $n \in \supp(x)$ (for example $m_1=1/6$, $m_2=5/12$), hence $\bar x \notin t^{\I}_\uu(\T)$ by Proposition \ref{Ghosh:Prop}.
\end{proof}

Moreover, the previous corollary has the following important consequence covering~\cite[Corollary 2.15]{Ghosh}.

\begin{corollary}
Let $\uu\in\A$. For any pair of translation invariant free $P$-ideals $\I_1, \I_2$ of $\N$, if $\I_2 \not \subseteq \I_1$, then $t^{\I_2}_\uu(\T) \not  \subseteq  t^{\I_1}_\uu(\T)$.
\end{corollary}
\begin{proof}
Let $B \in \I_2 \setminus \I_1$. As $B \notin \I_1,$ by Corollary \ref{Corollary2.14} there exists $x \in [0,1)$  with $\supp(x) \subseteq B$ such that $\bar x \notin t^{\I_1}_\uu(\T).$ But $\I_2$ is an ideal, hence $\supp(x) \in \I_2.$ Therefore, by Lemma \ref{Lemma2.2}, $\bar x \in t^{\I_2}_\uu(\T)$ and then $t^{\I_2}_\uu(\T) \not  \subseteq  t^{\I_1}_\uu(\T)$.
\end{proof}

\section{Appendix - a proof extending the idea from \cite{DI}}\label{appendix}

In this appendix, we see that one can make the proof from \cite{Ghosh} work, but under a strong condition on the ideal, that rules out even the ideal $\I_d$ of statistical convergence. On the other hand, the proof in this section covers the one in \cite{DI}, as $\F in$ satisfies the stronger condition (see Example~\ref{Laaaaast:Example}).

\subsection{The Claim}

One fundamental step in the proof of the sufficiency of Theorem~\ref{DiD} in~\cite{DiD,DI} is the counterpart of the following claim. In \cite{Ghosh} the author follows the same strategy as in \cite{DI} for the proof of his~\cite[Theorem~2.9]{Ghosh} (i.e., Theorem~\ref{Theorem2.9}), and~\cite[Claim~2.10]{Ghosh} is there the counterpart of the claim in~\cite{DI}. Here we rewrite \cite[Claim~2.10]{Ghosh} in our terms in Claim~\ref{Claim2.10}, and we apply it in the subsequent subsection in order to give a proof of the sufficiency in Theorem~\ref{conjecture} following the strategy from \cite{DiD,DI,Ghosh} under a suitable condition on the ideal that makes it work; then Theorem~\ref{Theorem2.9} follows from Theorem~\ref{conjecture} as explained in Remark~\ref{implies}.

For a more convenient use in the proof of the sufficiency, in the claim the necessary conditions $\ax$ and $\bx$ from Theorem~\ref{conjecture} are reformulated in a stronger iterated version. 

\begin{definition} 
Let $\uu\in\A$. For $A\in\B_\uu$ infinite, call $\beta(A):= \sup\{k\in \N: \mbox{$S_k(A)\in\B_\uu$}\}$ {\em degree of $b$-boundedness of} $A$. If $\beta(A)=  \infty$ (i.e., when $S_k(A)\in\B_\uu$ for all $k\in \N$) we say that  $A$ is \emph{infinitely $b$-bounded}.
\end{definition}

\begin{remark}\label{rem:beta}
Let $\uu\in\A$ and $A\in\B_\uu\setminus\mathcal Fin$.
\begin{itemize}
\item[(a)] If $A\subseteq A'\in\B_\uu\setminus\mathcal Fin$, then $\beta(A)\geq\beta(A')$.
\item[(b)] For $k\in\N$, $S_k(A)\in\B_\uu$ if and only if $A+i\in\B_\uu$ for every $i\in\{0,\ldots,k\}$.
\item[(c)] For $k\in\N$, $\beta(A)=k$ means $S_k(A)\in\B_\uu$ and $A+k+1\not\in\B_\uu$ (not 	necessarily $A+k+1\in\D_\uu$).
\item[(c$'$)] On the other hand, if %if $A$ %\NB\footnote{qui è un rifuso e serve dirlo anche dopo $A\in\B_\uu\setminus\mathcal Fin$?} is infinite, 
$S_k(A)\in\B_\uu$ and  $A+k+1\in\D_\uu$, then $\beta(A)=k$.
\end{itemize}
\end{remark}

\begin{remark}\label{rem:tail}
Let $\uu\in\A$, $x\in[0,1)$, $k\in\N$ and let $X\subseteq\N_+$ be infinite and such that $X+k+1\in\D_\uu$. Then
\begin{equation}\label{Eq:July21}
 \lim_{n\in X}\frac{\{u_{n+k+1}x\}}{b_n\cdots b_{n+k+1}}=0.
\end{equation}
Indeed, $\{u_{n+k+1}x\}<1$ and $b_i\geq1$ for every $i\in\N_+$, so that for $n\in X$, $0\leq\frac{\{u_{n+k+1}x\}}{b_n\cdots b_{n+k+1}} \leq \frac{1}{b_{n+k+1}}$,
and $\lim_{n\in X}\frac{1}{b_{n+k+1}}=\lim_{m\in X+k+1}\frac{1}{b_m}=0$, as $X+k+1$ is $b$-divergent.
\end{remark}
 
\begin{claim}\label{Claim2.10} 
Let $\uu\in\mathcal A$, let $\I$ be a translation invariant free ideal of $\N$ and let $x \in[0,1)$ with $S=\supp(x)$ and $S_b=\supp_b(x)$, such that $\ax$ holds.  Let $A\in\P(\N)\setminus \I$ and $k\in\N$ be such that $S_k(A)\in\B_\uu$. Then the following holds.
\begin{itemize}
\item[(i)] If $A \subseteq^{\I} S$, then $S_k(A) \subseteq^{\I} S_b$ and there exists  $A' \subseteq_{\I} A$ such that $S_k(A') \subseteq S_b$, $A'+k+1\subseteq S$ and $\lim_{n \in A'+k+1}\frac{c_n+1}{b_n}=1$. 

Moreover, if $A+k+1\in\D_\uu$, then $A'$ can be chosen to satisfy also 
$$\lim_{n\in A'+k+1}\frac{c_n}{b_n}=\lim_{n\in A'}\frac{c_{n+k+1}}{b_{n+k+1}}=1\quad \text{and}\quad\lim_{n\in A'}\frac{\{u_{n+k+1}x\}}{b_n\cdots b_{n+k+1}}=0.$$
\item[(ii)] If $A \cap S\in \I$, then $S_k(A) \cap S \in \I$ and there exists $B' \subseteq_{\I} A$ such that $S_k(B') \cap S = \emptyset$ and
$\lim_{n \in B'}\frac{c_{n+k+1}}{b_{n+k+1}}=0$. Moreover, if  $A^++k+1\in\D_\uu$ for some $A^+\subseteq A$, then $\lim_{n\in A^+}\frac{\{u_{n+k+1}x\}}{b_n\cdots b_{n+k+1}}=0.$
\end{itemize}
\end{claim}
\begin{proof}
(i) Assume that $A\asub S$. By Remark~\ref{rem:beta}(b), $A+i\in\B_\uu$ for every $i\in\{0,\dots,k\}$; moreover $A+i\notin\I$ for every $i\in\N$, since $A\notin\I$ and $\I$ is translation invariant.
 
We first prove, by induction on $i$, that $A+i\asub S$ for every $i\in\{0,\ldots,k+1\}$ and $A+i\asub S_b$ for every $i\in\{0,\ldots,k\}$. The
case $i=0$ is the hypothesis. Assume that $A+i\asub S$ for some $i\in\{0,\ldots,k\}$. As $A+i\in\B_\uu\setminus\I$, condition $\axuno$ applied to $A+i$ gives $A+i\asub S_b$ and $A+i+1\asub S$. Therefore:
\begin{itemize}
   \item[(1)] $A+i\asub S_b$ for every $i\in\{0,\dots,k\}$, hence $S_k(A)=\bigcup_{i=0}^{k}(A+i)\asub S_b$ by Lemma~\ref{New:Lemma}(a), and $A+k+1\asub S$.
\end{itemize}
 
Next, we apply again $\axuno$ to $A+k\in\B_\uu\setminus\I$, which satisfies $A+k\asub S$ by (1): there exists $A_1\lsub A+k$ with $A_1\subseteq S_b$ and $\lim_{m\in A_1}\frac{c_{m+1}+1}{b_{m+1}}=1$; by Remark~\ref{moreover} we may suppose that $A_1+1\subseteq S$. Since $\I$ is translation invariant,  $A^{(k+1)}:=A_1-k$ satisfies $A^{(k+1)}\lsub A$, and setting $n=m+1$ we obtain:
\begin{itemize}
  \item[(2)] $A^{(k+1)}\lsub A$, $\ A^{(k+1)}+k+1\subseteq S$ and $\lim\limits_{n\in A^{(k+1)}+k+1}\frac{c_n+1}{b_n}=1$.
\end{itemize}
 
We now produce $A''\lsub A$ with $S_k(A'')\subseteq S_b$. Fix $i\in\{0,\dots,k\}$. By (1) we have $A+i\asub S_b$, that is, $A\asub S_b-i$, and so $A\cap(S_b-i)\lsub A$ by Lemma~\ref{New:Lemma}(b). Therefore, Lemma~\ref{New:Lemma}(a) yields
$$A'':=A \cap \bigcap_{i=0}^k (S_b - i)=\bigcap_{i=0}^{k}\bigl(A\cap(S_b-i)\bigr)\lsub A .$$
In particular, $A''+i\subseteq S_b$ for every $i\in\{0,\dots,k\}$, i.e., $S_k(A'')\subseteq S_b$.
 
Taking $A':=A''\cap A^{(k+1)}$, we have $A'\lsub A$ by Lemma~\ref{New:Lemma}(a), and $A'\notin\I$ because $A\notin\I$. Since $A'\subseteq A''$,
we get $S_k(A')\subseteq S_k(A'')\subseteq S_b$, while $A'\subseteq A^{(k+1)}$ and (2) give $A'+k+1\subseteq S$ and $\lim_{n\in A'+k+1}\frac{c_n+1}{b_n}=1$.
 
 \medskip 
Assume now in addition that $A+k+1\in\D_\uu$ (so in particular $\beta(A)=k$ -- see Remark~\ref{rem:beta}(c$'$)). As $A'\subseteq A$, the infinite set $A'+k+1\subseteq A+k+1$ is $b$-divergent, so $\lim_{n\in A'+k+1}\frac{1}{b_n}=0$ and therefore
$$ 1=\lim_{n\in A'+k+1}\frac{c_n+1}{b_n } =\lim_{n\in A'+k+1}\Bigl(\frac{c_n}{b_n}+\frac{1}{b_n}\Bigr)=\lim_{n\in A'+k+1}\frac{c_n}{b_n} =\lim_{n\in A'}\frac{c_{n+k+1}}{b_{n+k+1}}.$$
Finally, since $A'+k+1\in\D_\uu$, we conclude with Remark \ref{rem:tail} that \eqref{Eq:July21} holds true.

\medskip 
(ii) Assume that $A\cap S\in\I$. As in (i), $A+i\in\B_\uu$ for every $i\in\{0,\dots,k\}$ by Remark~\ref{rem:beta}(b), and $A+i\notin\I$ for every $i\in\N$ by Lemma~\ref{ti1}.
 
\smallskip 
We first prove, by induction on $i$, that $(A+i)\cap S\in\I$ for every $i\in\{0,\dots,k\}$. The case $i=0$ is the hypothesis. 
Assume $(A+i)\cap S\in\I$ for some $i\in\{0,\dots,k-1\}$. Since $A+i\in\B_\uu\setminus\I$, condition $\axdue$ applied to $A+i$ provides $F\lsub A+i$ with $\lim_{n\in F}\frac{c_{n+1}}{b_{n+1}}=0$. As $i+1\leq k$, we have $F+1\subseteq A+i+1\in\B_\uu$, so there exists $M\in\N_+$ such that $b_{n+1}\leq M$ for every $n\in F$; consequently $c_{n+1}<\frac{b_{n+1}}{M}\leq1$, i.e., $c_{n+1}=0$, for almost every $n\in F$, which means that $(F+1)\cap S$ is finite. 
On the other hand $F\lsub A+i$ gives $F+1\lsub A+i+1$, and therefore
$$(A+i+1)\cap S \subseteq \bigl((A+i+1)\setminus(F+1)\bigr)\,\cup\,\bigl((F+1)\cap S\bigr) \in \I,$$
as $\I$ is free. This completes the induction, and consequently $S_k(A)\cap S=\bigcup_{i=0}^{k}\bigl((A+i)\cap S\bigr)\in\I$.

\smallskip  
Now $\axdue$ applied to $A+k$ gives $C\subseteq A$ such that $C+k\subseteq_\I A+k$ (so, $C\subseteq_\I A$ as $\I$ is translation invariant) and $\lim_{n\in C}\frac{c_{n+k+1}}{b_{n+k+1}}=0$. 
 
\smallskip 
We now produce $D\lsub A$ with $S_k(D)\cap S=\emptyset$. Fix $i\in\{0,\dots,k\}$. Then $(A+i)\cap S\in\I$ means $A\cap(S-i)\in\I$, so $D_i:=A\setminus(S-i)$ satisfies $D_i\lsub A$. Now Lemma~\ref{New:Lemma}(a) yields
$$D:=\bigcap_{i=0}^{k}D_i\lsub A.$$
Since $D\cap(S-i)=\emptyset$, i.e., $(D+i)\cap S=\emptyset$, for every $i\in\{0,\dots,k\}$, we conclude that $S_k(D)\cap S=\emptyset$.
 
Let $B':=C\cap D$. Then $B'\lsub A$ by Lemma~\ref{New:Lemma}(a), and $B'\notin\I$ because $A\notin\I$. Since $B'\subseteq D$ we get $S_k(B')\subseteq S_k(D)$, so $S_k(B')\cap S=\emptyset$; from $B'\subseteq C$ we conclude that $\lim_{n\in B'}\frac{c_{n+k+1}}{b_{n+k+1}}=0$.

\smallskip
Finally, let $A^{+}\subseteq A$ be such that $A^{+}+k+1\in\D_\uu$. Then the last assertion is Remark~\ref{rem:tail} applied to $X=A^{+}$.
\end{proof}

A first application of the above claim is the subsequent useful corollary. The following remark is needed in its proof. 

\begin{remark}\label{pseudo}
%\NBA\footnote{Anticipato e generalizzato}
Let $\I$ be a free $P$-ideal of $\N$, let $A\in\mathcal P(\N)\setminus\I$ and let $\{A^{(n)}:n\in\N\}$ be a family of subsets of $\N$ with 
$$A^{(n)}\subseteq_\I A\quad \text{for every $n\in\N$}.$$ 
Put $D_n:=A\setminus A^{(n)}\in\I$ for every $n\in\N$ and let $D\in\I$ be a pseudounion of $\{D_n:n\in\N\}$, that is, $D_n\subseteq^* D$ for every $n\in\N$. Then $B:=A\setminus D$ satisfies $B\subseteq_\I A$, as $A\setminus B\subseteq D\in\I$; moreover, $B\subseteq^\ast A\setminus D_n=A^{(n)}$ for every $n\in\N$, namely, $B\subseteq_\I A$ is a pseudointersection of $\{A^{(n)}:n\in\N\}$. Note that $B\notin\I$, in particular, $B$ is infinite.
\end{remark}

\begin{corollary}\label{L2.2} 
Let $\uu\in\mathcal A$, let $\I$ be a translation invariant free $P$-ideal of $\N$ and $x \in [0,1)$ with $S=\supp(x)$, $S_b=\supp_b(x)$ and satisfying $\ax$. If $A\subseteq \N$ is infinitely $b$-bounded and $A \not \in \I$, then there exists an infinite $B\subseteq A$ such that $\lim_{n \in B}\varphi(u_{n-1}x)=0$. 
\end{corollary}
\begin{proof} 
In order to build the infinite set $B\subseteq A$ with $\lim_{n \in B}\varphi(u_{n-1}x)=0$ we argue as follows. We need to show that for every fixed $k \in \N$ there exists a cofinite subset $C_k$ of $B$ such that $\|u_{n-1}x\| \leq \frac{1}{2^{k+1}}$ for every $n\in C_k$. 

 By Remark~\ref{case3} it is worth considering only the following two cases (note that every subset of $A$ is infinitely $b$-bounded by Remark~\ref{rem:beta}(a), so the possible replacement of $A$ made there is harmless), and we fix $k\in\N$.

\smallskip   
(1) If $A \subseteq^{\I} S$, then by Claim~\ref{Claim2.10}(i), there exists $A_k\subseteq_\I A$ with $S_k(A_k)\subseteq S_b$. In view of Lemma~\ref{ug10}(b), for every
$n \in A_k$, $\sigma_{n,k}(x)\geq 1-\frac{1}{2^{k+1}}$. This inequality and~\eqref{eqn:ug9} give $1-\frac{1}{2^{k+1}}\leq\sigma_{n,k}(x) \leq \{u_{n-1}x\}<1$ for every $n \in A_k$. Therefore, $\|u_{n-1}x\|=1-\{u_{n-1}x\}\leq\frac{1}{2^{k+1}}$ for every $n\in A_k$.

\medskip
(2) If $A \cap S \in \I$, then by Claim~\ref{Claim2.10}(ii), $S_k(A) \cap S \in \I$ and there exists $B_k\subseteq_{\I} A$ such that $S_k(B_k) \cap S = \emptyset$ and 
$\lim_{n \in B_k}\frac{c_{n+k+1}}{b_{n+k+1}}=0$. So, $\sigma_{n,k}(x) =0$ for all $n \in B_k$, in view of Lemma~\ref{ug10}(c), and 
$\frac{c_{n+k+1}}{b_{n+k+1}}< \frac{1}{2^{k+2}}$ % $\frac{1}{2^{k+1}}$ -> $\frac{1}{2^{k+2}}$
for almost all $n \in B_k$; 
%\NBA\footnote{Non indispensabile ma mi sembra meglio metterla.} 
replacing $B_k$ by a cofinite subset, which still satisfies $B_k\subseteq_\I A$ as $\I$ is free, we may assume that this holds for every $n\in B_k$.  Hence,~\eqref{eqn:ug9} implies that 
$$
\{u_{n-1}x\} < \sigma_{n,k}(x)+\frac{c_{n+k+1}}{b_n b_{n+1} \cdots b_{n+k+1}} +\frac{1}{2^{k+2}} = 
\frac{c_{n+k+1}}{b_n b_{n+1} \cdots b_{n+k+1}}+\frac{1}{2^{k+2}} 
\leq \frac{c_{n+k+1}}{b_{n+k+1}}+\frac{1}{2^{k+2}}
< \frac{1}{2^{k+1}}
$$
for every $n \in B_k$, and so $\|u_{n-1}x\|\leq\{u_{n-1}x\}<\frac{1}{2^{k+1}}$.

\smallskip
Carrying out this argument for every $k \in \N$ we find a countable family $\{A_k:k \in \N\}$ in case (1), resp., $\{B_k:k \in \N\}$ in case (2), of subsets of $A$
with $A_k\subseteq_{\I} A$, resp., $B_k\subseteq_{\I} A$, for every $k \in \N$. 
Since $\I$ is a $P$-ideal,  Remark~\ref{pseudo} provides a pseudointersection $A^*$, resp., $B^*$, of that family, such that $A^*\subseteq_{\I} A$, resp., $B^*\subseteq_{\I} A$;  in particular $A^*\subseteq^*A_k$, resp., $B^*\subseteq^* B_k$, for every $k\in\N$, and $A^*\notin\I$, resp., $B^*\notin\I$, so $A^*$, resp., $B^*$, is infinite. 

Let us see that $A^*$ (resp., $B^*$) is the needed set, i.e., $\lim_{n \in A^*}\varphi(u_{n-1}x)=0$ (resp., $\lim_{n \in B^*}\varphi(u_{n-1}x)=0$). Indeed, for a fixed $k \in \N$ the intersection $C_k:= A_k \cap A^*$ (resp., $C_k:=B_k\cap B^*$) is co-finite in $A^*$ (resp., in $B^*$), so obviously, $\|u_{n-1}x\| \leq \frac{1}{2^{k+1}}$ for every $n\in C_k$; as $k\in\N$ is arbitrary, this gives $\lim_{n \in A^*}\varphi(u_{n-1}x)=0$ (resp., $\lim_{n \in B^*}\varphi(u_{n-1}x)=0$), as desired. 
\end{proof}

\subsection{A proof of the sufficiency with the strategy from \cite{DiD,DI,Ghosh}}
 
Let $\uu\in\mathcal A$ and let $\I$ be a free $P$-ideal of $\N$. For $A \in\P(\N)\setminus \I$, if $A\setminus A'\not\in\I$ for every $A'\in\B_\uu\cap\P(A)$, then there exists $B\in\D_\uu\cap\P(A)$; indeed, since $\emptyset = A\setminus A\in\I$, $A$ cannot be $b$-bounded, therefore it has a $b$-divergent subset $B$. Unfortunately, one cannot achieve the conclusion $B\not\in\I$, which is needed in the proof that we give in this appendix of the sufficiency of the conditions $\ax$ and $\bx$ in the main theorem. This is why we introduce the following class of ideals: 

\begin{definition}\label{Def:DL}
Fixed $\uu\in\A$, a proper free ideal $\I$ of $\N$ satisfies:
\begin{itemize}
\item[$\DLu$] for every $A\in\P(\N)\setminus\I$, if $A\setminus A'\notin\I$ for every $A'\in\B_\uu\cap\P(A)$, then there exists $B\in(\D_\uu\cap\P(A))\setminus\I$ \\ 
(briefly, $A\not\in\B_\uu+\I$ implies $\D_\uu\cap\P(A)\not\subseteq\I$).
\end{itemize}
Moreover, we say that $\I$ satisfies $\DL$ if $\I$ satisfies $\DLu$ for every $\uu\in\A$. 
\end{definition}

%\begin{remark}\label{pseudo}
%Let $\I$ be a free $P$-ideal of $\N$. Assume that $A\in\P(\N)\setminus\I$ and 
%$$
%A=A^{(0)}{}_\I\supseteq A^{(1)}{}_\I\supseteq A^{(2)}{}_\I\supseteq\ldots{}_\I\supseteq A^{(n)}{}_\I\supseteq\ldots.
%$$
%Let $D\in\I$ be a pseudounion of $\{D_n:n\in\N\}$, where $D_n=A\setminus A^{(n)}\in\I$, that is, $D_n\subseteq^*D$ for every $n\in\N$. Let $B:=A\setminus D\subseteq_\I A$, as $A\setminus B\subseteq D\in\I$. Moreover, $B\subseteq^*A\setminus D_n=A^{(n)}$ for every $n\in\N$, namely, $B\subseteq_\I A$ is a  \emph{pseudointersection of} $\{A^{(n)}:n\in\N\}$.  Note that $B \not \in \I$, in particular, $B$ is infinite.
%\end{remark}

\begin{proposition}\label{Prop:6.6}\label{Lemma:June2}  
Let $\uu\in\A$ and let $\I$ be a translation invariant free $P$-ideal of $\mathbb{N}$ satisfying $\DLu$. Then for every $A\in\B_\uu\setminus\I$ at least one of the following holds:
\begin{itemize}
  \item[$(\alpha)$] there exists $A^{(\infty)}\lsub A$ with $\beta(A^{(\infty)})=\infty$;
  \item[$(\beta)$] there exist $k\in\mathbb{N}$ and $B\subseteq A$ with $B\notin\I$, $S_k(B)\in\B_\uu$ and $B+k+1\in\D_\uu$ \textup{(}hence $\beta(B)=k$\textup{)}.
\end{itemize}
\end{proposition}

\begin{proof}
We construct inductively a decreasing sequence
$A=A_0\supseteq A_1\supseteq A_2\supseteq\ldots$ such that, for every
$j\in\mathbb{N}$,
\begin{equation}\label{dagger}
A_j\lsub A,\quad A_j\notin\I,\quad S_j(A_j)\in\B_\uu.
\end{equation}
Let $A_0:=A$; then $S_0(A_0)=A\in\B_\uu$ by hypothesis, so \eqref{dagger} holds for $j=0$. Assume that $A_j$ satisfying \eqref{dagger} has been built and
put $X:=A_j+(j+1)$. Since $A_j\notin\I$ and $\I$ is translation invariant, $X\notin\I$. We distinguish two mutually exclusive cases.
 
\smallskip
\noindent\emph{Case 1: $X\setminus X'\in\I$ for some $X'\in\B_\uu\cap\mathcal{P}(X)$}.

Let $A_{j+1}:=X'-(j+1)$. Then $A_{j+1}\subseteq A_j$, as $X'\subseteq X=A_j+(j+1)$. Moreover $A_j\setminus A_{j+1}=(X\setminus X')-(j+1)\in \I$ by translation invariance, so $A_{j+1}\lsub A_j$, and $A_{j+1}\lsub A$ by Lemma~\ref{New:Lemma}(d); in particular $A_{j+1}\notin\I$, since $A\notin\I$. Finally,
$$S_{j+1}(A_{j+1})=S_j(A_{j+1})\cup\bigl(A_{j+1}+j+1\bigr)\subseteq S_j(A_j)\cup X' \in\B_\uu,$$
because $A_{j+1}\subseteq A_j$ and $\B_\uu$ is an ideal. Thus \eqref{dagger} holds for $j+1$ and the construction goes on.
 
\smallskip
\noindent\emph{Case 2: $X\setminus X'\notin\I$ for every $X'\in\B_\uu\cap\mathcal{P}(X)$}. 

This is precisely the hypothesis of $\DLu$ for the set $X\notin\I$, so there exists $C\in(\D_\uu\cap\mathcal{P}(X))\setminus\I$. 
Let $B:=C-(j+1)$. Then $B\subseteq X-(j+1)=A_j\subseteq A$ and $B\notin\I$, since $C=B+(j+1)\notin\I$ and $\I$ is translation invariant. Moreover, $B+j+1=C\in\D_\uu$ and
$S_j(B)\subseteq S_j(A_j)\in\B_\uu$. Hence, $(\beta)$ holds with $k=j$.
 
\smallskip
If Case 2 occurs for some $j$, $(\beta)$ holds. Otherwise Case 1 applies for every $j\in\N$ and provides a sequence $(A_j)$ such that for every $j\in\mathbb{N}$, $A_j$ satisfies \eqref{dagger}. Since $A\setminus A_j\in\I$ for every $j\in\mathbb{N}$, by Remark~\ref{pseudo} there exists a pseudointersection $A^{(\infty)}\lsub A$ of $\{A_j: j\in\mathbb{N}\}$ with $A^{(\infty)}\notin\I$; in particular, $A^{(\infty)}\subseteq^{*}A_j$ for every $j\in\mathbb{N}$. Fix $j\in\mathbb{N}$. Then $A^{(\infty)}+i\subseteq^{*}A_j+i$ for every $i\in\{0,\dots,j\}$, hence $S_j(A^{(\infty)})\subseteq^{*}S_j(A_j)$. As $S_j(A_j)\in\B_\uu$, $\mathcal Fin\subseteq\B_\uu$ and $\B_\uu$ is an ideal, we conclude that $S_j(A^{(\infty)})\in\B_\uu$.
Since $j\in\mathbb{N}$ is arbitrary, $\beta(A^{(\infty)})=\infty$, that is, $(\alpha)$ holds.
\end{proof}

\begin{theorem}[Sufficiency under $\DLu$]\label{SuffiDL} Let $\uu\in\mathcal A$, let $\I$ be a translation invariant free $P$-ideal of $\N$ and $x \in[0,1)$ with $S=\supp(x)$ and $S_b=\supp_b(x)$. If $\I$ satisfies $\DLu$ and $\ax\&\bx$ holds, then $\bar x\in t_\uu^\I(\T)$.
\end{theorem} 
\begin{proof}
To prove that $\bar x \in t_\uu^{\I}(\T)$, by Lemma~\ref{UAholds} it is sufficient to show that $(U_A)$ holds for every $A \in\P(\N)\setminus \I$,
i.e., that there exists an infinite $B' \subseteq A$ such that $\lim_{n \in B'}\varphi(u_{n-1}x)=0$, and this is what we verify in each of the cases below. Fix $A \in\P(\N)\setminus \I$.
 
 \medskip
{\bf (i)} If there exists $C\in\B_\uu\cap\P(A)\setminus\I$, then, since $(U_C)$ implies $(U_A)$, we can replace $A$ by $C$ and assume without loss of generality that $A\in\B_\uu$. Moreover, in view of Remark~\ref{case3}, we can assume without loss of generality that either $A\subseteq S$ or $A\cap S=\emptyset$; note that both properties are inherited by every subset of $A$, and that $A$ remains in $\B_\uu\setminus\I$ after this replacement. By Proposition~\ref{Lemma:June2} applied to $A\in\B_\uu\setminus\I$, at least one of the following two cases occurs.

 \medskip
 
{\bf (i1)} There exists $A'\subseteq_\I A$ with $\beta(A') = \infty$. As $A\not\in\I$, also $A'\not\in\I$, so $\lim_{n \in B}\varphi(u_{n-1}x)=0$ for some infinite $B\subseteq A'\subseteq A$, by Corollary~\ref{L2.2}.
 
  \medskip

{\bf (i2)} There exist $k\in \N$ and $B \subseteq A$ such that $B \not \in \I$, $S_k(B)\in\B_\uu$ and $B+k+1\in\D_\uu$; in particular $\beta(B)=k$. By Remark~\ref{rem:tail} applied to $X=B$,
\begin{equation} \label{eqn:ug12}
\lim_{n \in B}\frac{\{u_{n+k+1}x\}}{b_nb_{n+1} \cdots b_{n+k+1}}=0.
\end{equation}
From~\eqref{eqn:ug8} and~\eqref{eqn:ug12} it follows that
\begin{equation}\label{eqn:ug13}
\lim_{n \in B'}\left(\{u_{n-1}x\}-\sigma_{n,k}(x)-\frac{c_{n+k+1}}{b_n b_{n+1} \cdots b_{n+k+1}}\right)=0 \quad\text{for every infinite } B'\subseteq B .
\end{equation}
In particular, for every infinite $B'\subseteq B$ the limit $\lim_{n \in B'}\{u_{n-1}x\}$ exists if and only if $\lim_{n \in B'}\left(\sigma_{n,k}(x)+\frac{c_{n+k+1}}{b_n b_{n+1} \cdots b_{n+k+1}}\right)$ does, and in that case the two limits coincide.
 
 \smallskip
By the reduction made in {\bf (i)}, it suffices to consider only the following two subcases.
 
 \smallskip
{\bf (i2.1)} Assume that $A \subseteq S$, so $B \subseteq^{\I} S$ as well. In view of Claim~\ref{Claim2.10}(i) applied to $B\not\in\I$, for which $S_k(B)\in\B_\uu$ and $B+k+1\in\D_\uu$, there exists $B' \subseteq_{\I} B$ with $S_k(B')\subseteq S_b$ and, by the last assertion of Claim~\ref{Claim2.10}(i), $\lim_{n \in B'}\frac{c_{n+k+1}}{b_{n+k+1}}=1$, that is $\lim_{n \in B'+k+1}\frac{c_n}{b_n}=1$. As $B\not\in\I$, also $B'\not\in\I$, so $B'$ is infinite. Since $S_k(B')\subseteq S_b$, for every $n\in B'$ one has $n,n+1,\dots,n+k\in S_b$, so~\eqref{eqn:ug10} gives $\sigma_{n,k}(x)=1-\frac{1}{b_nb_{n+1}\cdots b_{n+k}}$. Therefore
\begin{equation*}
\lim_{n \in B'}\left(\sigma_{n,k}(x)+\frac{c_{n+k+1}}{b_nb_{n+1} \cdots b_{n+k+1}}\right)=\lim_{n \in B'}\left(1+ \frac{1}{b_n b_{n+1} \cdots b_{n+k}}  \left(\frac{c_{n+k+1}}{ b_{n+k+1}}-1  \right) \right)=1,
\end{equation*}
as $0<\frac{1}{b_n b_{n+1} \cdots b_{n+k}}\leq1$ and $\frac{c_{n+k+1}}{ b_{n+k+1}}\to1$ along $B'$. Hence, $\lim_{n \in B'}\{u_{n-1}x\}=1$ by~\eqref{eqn:ug13}, and so $\lim_{n \in B'}\varphi(u_{n-1}x)=0$.
 
\medskip
{\bf (i2.2)} Assume that $A \cap S = \emptyset$, so $B\cap S\in\I$ as well.  In view of Claim~\ref{Claim2.10}(ii) applied to $B\not\in\I$, for which
$S_k(B)\in\B_\uu$ and $B\cap S\in\I$, there exists $B'\subseteq_\I B$ such that $S_k(B')\cap S=\emptyset$ and $\lim_{n \in B'}\frac{c_{n+k+1}}{b_{n+k+1}}=0$; as above $B'\not\in\I$, so $B'$ is infinite. Since $S_k(B')\cap S=\emptyset$, for every $n\in B'$ one has $n,n+1,\dots,n+k\not\in S$, so Lemma~2.7(c) gives $\sigma_{n,k}(x)=0$. Therefore, $\lim_{n \in B'}\varphi(u_{n-1}x)=0$, since by~\eqref{eqn:ug13}
$$\lim_{n \in B'}\{u_{n-1}x\}=\lim_{n \in B'}\left( \sigma_{n,k}(x)+\frac{c_{n+k+1}}{b_n b_{n+1} \cdots b_{n+k+1}} \right) =
\lim_{n \in B'}  \frac{c_{n+k+1}}{b_n b_{n+1} \cdots b_{n+k+1}} =0,$$
as $0\leq\frac{c_{n+k+1}}{b_n b_{n+1} \cdots b_{n+k+1}}\leq\frac{c_{n+k+1}}{ b_{n+k+1}}$ for every $n\in B'$.
 
\medskip
{\bf (ii)} If $\B_\uu\cap\P(A)\subseteq\I$, by Lemma~\ref{Lemma2.8*}, there exists $B\in\D_\uu$ with $B\subseteq_\I A$; as $A\not\in\I$, also $B\not\in\I$.
Then Lemma~\ref{L2.3} gives $B' \subseteq_\I B$ such that $\lim_{n \in B'}\varphi(u_{n-1}x)=0$, and $B'\not\in\I$ is infinite.
 
\smallskip
Resuming, we proved in (i) and (ii) that for any $A\in\P(\N)\setminus\I$, there exists an infinite set $B' \subseteq A$ such that $\lim_{n \in B'}\varphi(u_{n-1}x)=0$, i.e., that $(U_A)$ holds. This implies that $\bar x \in t_\uu^{\I}(\T)$ by Lemma~\ref{UAholds}.
\end{proof}

In the following lemma  (used in Example \ref{Laaaaast:Example}) we find a condition  equivalent to $\DL$ that does not involve any sequence. 
  
\begin{lemma}\label{lemma2}
A proper free ideal $\I$ of $\N$ satisfies $\DL$ if and only if for any increasing sequence $(A_n)$ in $\P(\N)$ with $A_n\uparrow A$ such that $A\setminus A_n\notin\I$ for every $n\in\N$, there exists a pseudounion $U$ of $(A_n)$, such that $U^*\subseteq A$ and $U^* \not \in \I$ (equivalently, there exists $B\subseteq A$ such that $B\notin \I$ and $B\cap A_n$ is finite for all $n\in\N$).
\end{lemma}
\begin{proof} 
Assume that $\I$ satisfies $\DL$, and let $(A_n)$ in $\P(\N)$ be an increasing sequence with $A_n\uparrow A$ such that $A\setminus A_n\notin\I$ for every $n\in\N$. We can assume without loss of generality that $A_0=\emptyset$ and the sequence $(A_n)$ is strictly increasing.  Define $b_k=n$ for $k\in A_n\setminus A_{n-1}$ and for $n\geq 2$, and $b_k=2$ for $k\in\N\setminus A$ and $k\in A_1$; moreover, let $u_0=1$, $u_n=\prod_{k=1}^nb_k$ for every $n\in\N_+$ and $\uu=(u_n)$. If $A'\subseteq A$ is $b$-bounded, then $A'\subseteq A_m$ for some $m\in\N$. Therefore, $A\setminus A'\supseteq A\setminus A_m$, moreover $A\setminus A_m\notin\I$; thus $A\setminus A'\notin \I$. By $\DLu$, there exists a $b$-divergent set $B\subseteq A$ with $B\notin\I$. Then $B\cap A_n$ is finite for all $n\in\N$.
Hence, $A_n\subseteq ^*B^*$ for all $n\in\N$ implies that $U=B^*$ is a pseudounion of $(A_n)$.

Vice versa, let $\uu\in\mathcal A$ and $A\in\P(\N)\setminus\I$ such that $A\setminus A'\notin\I$ for every $b$-bounded $A'\subseteq A$. Moreover, let $A_n=\{k\in A:b_k\leq n\}$ for every $n\in\N$.  By hypothesis, there exists $B\subseteq A$ such that $B\notin \I$ and $B\cap A_n$ is finite for all $n\in\N$. This $B$ is $b$-divergent and $B\notin \I$ as required in $\DLu$.
\end{proof}

\begin{example}\label{Laaaaast:Example}
(a) $\F in$ satisfies $\DL$. Indeed, let $(A_n)$ be an increasing sequence in $\N$ and $A_n\uparrow A$ such that $A\setminus A_n\notin\mathcal Fin$ for every $n\in\N$.
We can assume without loss of generality that $(A_n)$ is strictly increasing. For every $n\in\N$, let $m_n\in A_{n+1}\setminus A_{n}$. Then $B=\{m_n:n\in\N\}$ is infinite with $B\cap A_n$ finite for all $n\in\N$. So, we apply Lemma~\ref{lemma2}.

\smallskip
(b) The density ideal $\I_d$ does not satisfy $\DL$. According to Lemma \ref{lemma2}, 
it suffices to verify that there exists an increasing sequence $A_n\subseteq\N$ with $A_n\uparrow A$ and $A\setminus A_n\notin\I_d$ for every $n\in\N$. For $n\in \N$ let $B_n=2^n\N$ and $A_n=B_n^*$. 
Then $A_n\uparrow\N_+$ and $d(A_n^*) = d(B_n) >0$ (so $A_n^*\not \in \I_d$) for every $n\in \N$. Let $A=\N$ and let $U$ be a pseudounion of $(A_n)$. 
Then $A_n\subseteq^* U$, therefore $U^*\subseteq^* A_n^*=B_n$, and so $d(U^*)\leq d(B_n)\rightarrow 0$. Hence, $d(U^*)=0$, i.e.,  $U^*\in\I_d$.
\end{example}

\end{document}